\documentclass{amsart}

\newtheorem{theorem}[equation]{Theorem}
\newtheorem{lemma}[equation]{Lemma}
\newtheorem{corollary}[equation]{Corollary}
\newtheorem{proposition}[equation]{Proposition}

\numberwithin{equation}{section}

\usepackage{amscd,amsmath,amssymb,verbatim}

\begin{document}

\title[A $p$-adic refinement of the Hasse-Witt matrix]{$A$-hypergeometric series and a $p$-adic refinement of the Hasse-Witt matrix}
\author{Alan Adolphson}
\address{Department of Mathematics\\
Oklahoma State University\\
Stillwater, Oklahoma 74078}
\email{adolphs@math.okstate.edu}
\author{Steven Sperber}
\address{School of Mathematics\\
University of Minnesota\\
Minneapolis, Minnesota 55455}
\email{sperber@math.umn.edu}
\date{\today}
\begin{abstract}
We identify the $p$-adic unit roots of the zeta function of a projective hypersurface over a finite field of characteristic $p$ as the eigenvalues of a product of special values of a certain matrix of $p$-adic series.  That matrix is a product $F(\Lambda^p)^{-1}F(\Lambda)$, where the entries in the matrix $F(\Lambda)$ are $A$-hypergeometric series with integral coefficients and $F(\Lambda)$ is independent of~$p$.
\end{abstract}
\maketitle

\section{Introduction}

Dwork\cite{D1} expressed the unit root of the zeta function of an ordinary elliptic curve in the Legendre family in characteristic~$p$ in terms of the Gaussian hypergeometric function ${}_2F_1(1/2,1/2,1,\lambda)$.  There have since been a number of generalizations and extensions of that result:  see the Introduction to \cite{AS1}.  The point of the current paper is to prove such a result for projective hypersurfaces where the zeta function has multiple unit roots by using the matrix of $A$-hypergeometric series that appeared in~\cite{AS0}.  We note that similar results have recently been obtained for hypersurfaces in the torus by Beukers and Vlasenko\cite{BV1,BV2} using different methods.

Let $\{x^{{\bf b}_k}\}_{k=1}^N$ be the set of {\em all} monomials of degree $d$ in variables $x_0,\dots,x_n$ \big(so $N=\binom{d+n}{n}$\big) and let
\begin{equation}
f_\lambda(x_0,\dots,x_n) =  \sum_{k=1}^N \lambda_k x^{{\bf b}_k}\in {\mathbb F}_q[x_0,\dots,x_n]
\end{equation}
be a homogeneous polynomial of degree $d$ over the finite field ${\mathbb F}_q$, $q=p^a$, $p$ a prime.  Let $X_{\lambda}\subseteq{\mathbb P}^n_{{\mathbb F}_q}$ be the projective hypersurface defined by the vanishing of $f_\lambda$ and let $Z(X_\lambda/{\mathbb F}_q,t)$ be its zeta function.  Define a rational function $P_\lambda(t)$ by the equation
\[ P_\lambda(t) = \big(Z(X_\lambda/{\mathbb F}_q,t)(1-t)(1-qt)\cdots(1-q^{n-1}t)\big)^{(-1)^n}\in 1+t{\mathbb Z}[[t]]. \]
When $X_\lambda$ is smooth, $P_\lambda(t)$ is the characteristic polynomial of Frobenius acting on middle-dimensional primitive cohomology.  In this case, $P_\lambda(t)$ has degree 
\[ d^{-1}\big((d-1)^{n+1}+(-1)^{n+1}(d-1)\big). \]
In the general case, we know only that $P_\lambda(t)$ is a rational function.  

We regard $f_\lambda$ as a polynomial with fixed exponents and variable coefficients, giving rise to a family of rational functions $P_\lambda(t)$.  The $p$-adic unit (reciprocal) roots of $P_\lambda(t)$ all occur in the numerator (Proposition 1.3) and for generic $\lambda$ it has the maximal possible number of $p$-adic unit (reciprocal) roots (by the generic invertibility of the Hasse-Witt matrix, see below).  Our goal in this paper is to give a $p$-adic analytic formula for these unit roots in terms of $A$-hypergeometric series.

Write ${\bf b}_k = (b_{0k},\dots,b_{nk})$ with $\sum_{i=0}^n b_{ik} = d$.  It will be convenient to define an augmentation of these vectors.
For $k=1,\dots,N$, put
\[ {\bf a}_k = ({\bf b}_k,1)  = (b_{0k},b_{1k},\dots,b_{nk},1)\in{\mathbb N}^{n+2} \]
(where ${\mathbb N}$ denotes the nonnegative integers).  Let $A = \{{\bf a}_k\}_{k=1}^N$.  Note that the vectors ${\bf a}_k$ all lie on the hyperplane $\sum_{i=0}^n u_i=du_{n+1}$ in ${\mathbb R}^{n+2}$.  Let
\[ U = \bigg\{ u=(u_0,\dots,u_n,1)\in{\mathbb N}^{n+2}\mid \text{$\sum_{i=0}^n u_i=d$ and $u_i>0$ for $i=0,\dots,n$}\bigg\}. \]
Note that $\{x_0^{u_0}\cdots x_n^{u_n}\mid u\in U\}$ is the set of all monomials of degree $d$ that are divisible by the product $x_0\cdots x_n$, so $\lvert U \rvert = \binom{d-1}{n}$.  We assume throughout that $d\geq n+1$, which implies $U\neq\emptyset$.  (If $d<n+1$, then $U$ is empty and none of the reciprocal roots of $P_\lambda(t)$ is a $p$-adic unit.  If one defines the determinant of an empty matrix to be~1, then Theorem~1.10 below is trivially true in this case.)  

We recall the definition of the Hasse-Witt matrix of $f_\lambda$, a matrix with rows and columns indexed by $U$.  Let $\Lambda_1,\dots,\Lambda_N$ be indeterminates and let $H(\Lambda) = \big[ H_{uv}(\Lambda)\big]_{u,v\in U}$ be the matrix of polynomials:
\begin{equation}
H_{uv}(\Lambda) = \sum_{\substack{\nu\in{\mathbb N}^N\\ \sum_{k=1}^N \nu_k{\bf a}_k = pu-v}} \frac{\Lambda_1^{\nu_1}\cdots \Lambda_N^{\nu_N}}{\nu_1!\cdots\nu_N!}\in{\mathbb Q}[\Lambda_1,\dots,\Lambda_N].
\end{equation}
Since the last coordinate of each ${\bf a}_k$, $u$, and $v$ equals $1$, the condition on the summation implies $\sum_{k=1}^N \nu_k = p-1$.
In particular, it follows that $\nu_k\leq p-1$ for all $k$, so the coefficients of each $H_{uv}(\Lambda)$ lie in ${\mathbb Q}\cap{\mathbb Z}_p$.
Let $\bar{H}_{uv}(\Lambda)\in{\mathbb F}_p[\Lambda]$ be the reduction mod $p$ of $H_{uv}(\Lambda)$ and let $\bar{H}(\Lambda)$ be the reduction mod $p$ of $H(\Lambda)$.  Then $\bar{H}(\lambda)$ is the Hasse-Witt matrix of $f_\lambda$ (relative to a certain basis: see Katz\cite[Algorithm 2.3.7.14]{K}).

Write $P_\lambda(t) = Q_\lambda(t)/R_\lambda(t)$, where $Q_\lambda(t),R_\lambda(t)\in 1 + {\mathbb Z}[t]$ are relatively prime polynomials.  As a special case of \cite[Theorem 1.4]{AS} we have the following result.
\begin{proposition}
For all $\lambda\in{\mathbb F}_q^N$ we have $R_\lambda(t)\equiv 1 \pmod{q}$ and
\[ Q_\lambda(t) \equiv \det\big( I-t \bar{H}(\lambda^{p^{a-1}}) \bar{H}(\lambda^{p^{a-2}})\cdots \bar{H}(\lambda)\big) \pmod{p}. \]
\end{proposition}

Proposition 1.3 implies that all the unit roots of $P_\lambda(t)$ occur in the numerator and that there are at most ${\rm card}(U)$ of them.  The Hasse-Witt matrix is known to be generically invertible for a ``sufficiently general'' polynomial $f_\lambda$ (Koblitz\cite{Ko}, Miller\cite{M1,M2}).  We recall the precise version of that fact that we need.  

We suppose the ${\bf a}_k$ are ordered so that $U=\{{\bf a}_k\}_{k=1}^M$, $M=\binom{d-1}{n}$.  We may then write the matrix $H(\Lambda)$ as $\big[ H_{ij}(\Lambda)\big]_{i,j=1}^M$, where
\begin{equation}
H_{ij}(\Lambda) = \sum_{\substack{\nu\in{\mathbb N}^N\\ \sum_{k=1}^N \nu_k{\bf a}_k = p{\bf a}_i-{\bf a}_j}} \frac{\Lambda_1^{\nu_1}\cdots \Lambda_N^{\nu_N}}{\nu_1!\cdots\nu_N!}\in({\mathbb Q}\cap{\mathbb Z}_p)[\Lambda_1,\dots,\Lambda_N].
\end{equation}

Consider the related matrix  $B(\Lambda) = \big[B_{ij}(\Lambda)\big]_{i,j=1}^M$ of Laurent polynomials defined by
\begin{equation}
B_{ij}(\Lambda) = \Lambda_i^{-p}\Lambda_j H_{ij}(\Lambda)\in ({\mathbb Q}\cap{\mathbb Z}_p)[\Lambda_1,\dots,\Lambda_{i-1},\Lambda_i^{-1},\Lambda_{i+1},\dots,\Lambda_N],
\end{equation}
i.~e., $B(\Lambda) = C(\Lambda^p)^{-1}H(\Lambda)C(\Lambda)$, where $C(\Lambda)$ is the diagonal matrix with entries $\Lambda_1,\dots,\Lambda_M$.  
Since $\lambda_k^{p^a}=\lambda_k$ for $\lambda_k\in{\mathbb F}_q$, Proposition 1.3 has the following corollary.
\begin{corollary}
For all $\lambda\in({\mathbb F}_q^\times)^M\times {\mathbb F}_q^{N-M}$, 
\[  Q_\lambda(t) \equiv \det\big( I-t \bar{B}(\lambda^{p^{a-1}}) \bar{B}(\lambda^{p^{a-2}})\cdots \bar{B}(\lambda)\big) \pmod{p}. \]
\end{corollary}

Put $D(\Lambda) = \det B(\Lambda)$.  By \cite[Proposition 2.11]{AS0} we have the following result.
\begin{proposition}
The Laurent polynomial $D(\Lambda)$ has constant term $1$.
\end{proposition}

Proposition 1.7 implies that neither $D(\Lambda)$ nor $\bar{D}(\Lambda)$ is the zero polynomial, so the matrix $\bar{B}(\lambda)$ is invertible for $\lambda$ in a Zariski open subset of $({\mathbb F}_q^\times)^M\times {\mathbb F}_q^{N-M}$, a nonempty set for $q$ sufficiently large.  It follows from Corollary~1.6 that for $\bar{D}(\lambda)\neq 0$, $P_\lambda(t)$ has $M$ unit reciprocal roots.  Let $\{\pi_j(\lambda)\}_{j=1}^M$ be these roots and put
\[ \rho(\lambda,t) =\prod_{j=1}^M \big(1-\pi_j(\lambda)t\big). \]
Our goal is to describe the polynomial $\rho(\lambda,t)$ in terms of the $A$-hypergeometric series introduced in \cite[Equation~(3.3)]{AS0}.  

We begin by defining a ring that contains the desired series.  
For each $i=1,\dots,M$, the Laurent polynomials $B_{ij}(\Lambda)$ have exponents lying in the set
\[ L_{i} := \bigg\{ l=(l_1,\dots,l_{N})\in{\mathbb Z}^{N}\mid \text{$\sum_{k=1}^{N} l_k{\bf a}_k = {\bf 0}$, $l_i\leq 0$, and $l_k\geq 0$ for $k\neq i$}\bigg\}. \]
Let ${\mathcal C}\subseteq{\mathbb R}^{N}$ be the real cone generated by $\bigcup_{i=1}^M L_{i}$.  By \cite[Proposition 2.9]{AS0}, the cone ${\mathcal C}$ does not contain a line, hence it has a vertex at the origin.  
Put $E={\mathcal C}\cap{\mathbb Z}^N$.  Note that $(l_1,\dots,l_N)\in E$ implies that $\sum_{k=1}^N l_k{\bf a}_k = {\bf 0}$ and that $\sum_{k=1}^N l_k = 0$ since the last coordinate of each ${\bf a}_k$ equals~1.  
Let ${\mathbb C}_p$ be the completion of an algebraic closure of ${\mathbb Q}_p$ with absolute value $\lvert\cdot\rvert$ associated to the valuation ${\rm ord}$, normalized by  ${\rm ord}\;p=1$ and $\lvert p\rvert = p^{-1}$.  Set
\[ R_E=\bigg\{ \xi(\Lambda) = \sum_{l=(l_1,\dots,l_N)\in E} c_l\Lambda_1^{l_1}\cdots\Lambda_N^{l_N}\mid \text{$c_l\in{\mathbb C}_p$ and $\{\lvert c_l\vert\}_{l\in E}$ is bounded}\bigg\}, \]
i.~e., $R_E$ is the set of Laurent series over ${\mathbb C}_p$ having exponents in $E$ and bounded coefficients.  
Note that $R_E$ is a ring (since ${\mathcal C}$ has a vertex at the origin) and that the entries of $B(\Lambda)$ all lie in $R_E$.  By Proposition 1.7, the Laurent polynomial $D(\Lambda)$ is an invertible element of $R_E$, so $B(\Lambda)^{-1}$ has entries in $R_E$.

We define a matrix $F(\Lambda) = \big[ F_{ij}(\Lambda)\big]_{i,j=1}^M$ with entries in $R_E$.  For $i\neq j$, put
\[ F_{ij}(\Lambda) = \sum_{\substack{(l_1,\dots,l_N)\in L_i\\ l_j>0}} (-1)^{-l_i-1} \frac{(-l_i-1) !}{\displaystyle (l_j-1)!\prod_{\substack{k=1\\ k\neq i,j}}^N l_k!} \Lambda_1^{l_1}\cdots \Lambda_N^{l_N}, \]
and for $i=j$ put
\[ F_{ii}(\Lambda) = \sum_{(l_1,\dots,l_N)\in L_i} (-1)^{-l_i} \frac{(-l_i)!}{\displaystyle \prod_{\substack{k=1\\ k\neq i}}^N l_k!} \Lambda_1^{l_1}\cdots \Lambda_N^{l_N}. \]
The coefficients of these series are multinomial coefficients, hence lie in~${\mathbb Z}$, so these series lie in $R_E$.
Note also that each $F_{ii}(\Lambda)$ has constant term $1$, while the $F_{ij}(\Lambda)$ for $i\neq j$ have no constant term.  It follows that $\det F(\Lambda)$ has constant term $1$, hence $\det F(\Lambda)$ is an invertible element of $R_E$.  We may therefore define a matrix ${\mathcal F}(\Lambda)$ with entries in $R_E$ by the formula
\begin{equation}
 {\mathcal F}(\Lambda) = F(\Lambda^p)^{-1}F(\Lambda).
\end{equation}
The interpretation of the $F_{ij}(\Lambda)$ as $A$-hypergeometric series will be explained in Section 7 (see Eq.~(7.12)).

Put
\begin{multline}
 {\mathcal D} = \{ (\lambda_1,\dots,\lambda_N)\in {\mathbb C}_p^N\mid \text{$\lvert\lambda_k\rvert=1$ for $k=1,\dots,M$,} \\ \text{$\lvert\lambda_k\rvert\leq 1$ for $k=M+1,\dots,N$, and $\lvert D(\lambda)\rvert=1$}\}.
\end{multline}
Our main result is the following.
\begin{theorem}
The entries in the matrix ${\mathcal F}(\Lambda)$ are functions on ${\mathcal D}$.  
Let $\lambda\in ({\mathbb F}_q^\times)^M\times{\mathbb F}_q^{N-M}$ and let $\hat{\lambda}\in{\mathbb Q}_p(\zeta_{q-1})^N$ be its Teichm\"uller lifting.  If $\bar{D}(\lambda)\neq 0$, then $\hat{\lambda}^{p^i}\in{\mathcal D}$ for $i=0,\dots,a-1$ and 
\[ \rho(\lambda,t) = \det\big( I-t{\mathcal F}(\hat{\lambda}^{p^{a-1}}){\mathcal F}(\hat{\lambda}^{p^{a-1}})\cdots{\mathcal F}(\hat{\lambda})\big). \]
\end{theorem}

{\bf Remark 1.}  The first sentence of Theorem 1.10 will be made more precise in Section 2 (see Theorem 2.11).  The assertion that $\hat{\lambda}^{p^i}\in{\mathcal D}$ for all $i$ if $\bar{D}(\lambda)\neq 0$ follows immediately from the fact that 
$\bar{D}(\Lambda)$ has coefficients in ${\mathbb F}_p$.

{\bf Remark 2.}  The series $F_{ij}(\Lambda)$ are related to the Laurent polynomials $B_{ij}(\Lambda)$ by truncation (see \cite[Proposition 3.8]{AS0}).  One can thus regard Theorem 1.10 as a refinement of Corollary 1.6, i.~e., one may think of 
${\mathcal F}(\Lambda)$ as a $p$-adic refinement of the Hasse-Witt matrix $B(\Lambda)$.  

\section{Rings of $p$-adic series}

This section has two purposes.  We need a ring $R$ large enough to contain all the series that will be encountered in the proof of Theorem~1.10 and we need to identify a subring $R'\subseteq R$ whose elements define functions on ${\mathcal D}$.  

Consider the ${\mathbb C}_p$-vector space 
\[ {\mathbb B}:= \bigg\{ \xi(\Lambda) = \sum_{l\in{\mathbb Z}^N} c_l\Lambda^l\mid \text{$\{\lvert c_l\rvert\}_{l\in{\mathbb Z}^N}$ is bounded} \bigg\} \]
of all Laurent series with bounded coefficients.  We define a norm on ${\mathbb B}$ by setting
\[ \lvert \xi(\Lambda)\rvert = \sup_{l\in{\mathbb Z}^N} \lvert c_l\rvert. \]
The ${\mathbb C}_p$-vector space ${\mathbb B}$ is complete in this norm and $R_E$ is a ${\mathbb C}_p$-subspace of ${\mathbb B}$ which is also complete.  

We begin by defining a subring $R'_E$ of the ring $R_E$ whose elements define functions on ${\mathcal D}$.  
Let $L_E$ denote the subring of $R_E$ consisting of the Laurent polynomials that lie in $R_E$.  Since $D(\Lambda)$ is an invertible element of $R_E$, we may define a subring of~$R_E$
\[  L_{E,D(\Lambda)} := \bigg\{ \frac{h(\Lambda)}{D(\Lambda)^k}\mid \text{$h(\Lambda)\in L_E$ and $k\in{\mathbb N}$}\bigg\}.  \]
The definitions of $R_E$ and ${\mathcal D}$ show that the elements of $L_{E,D(\Lambda)}$ are functions on~${\mathcal D}$.  
We define $R_E'$ to be the completion of $L_{E,D(\Lambda)}$ under the norm on ${\mathbb B}$.  Clearly $R_E'$ is a subring of $R_E$.  The next two lemmas will show that the elements of $R_E'$ define functions on ${\mathcal D}$.

\begin{lemma}
For $h(\Lambda)\in L_E$ and $k\in{\mathbb N}$, one has
\[ \bigg\lvert\frac{h(\Lambda)}{D(\Lambda)^k}\bigg\rvert = \lvert h(\Lambda)\rvert. \]
\end{lemma}

\begin{proof}
Write 
\[ D(\Lambda)^{-k} = \sum_{l\in E} d_l\Lambda^l. \]
Since $D(\Lambda)$ has constant term 1 and $p$-integral coefficients, it follows that $d_{\bf 0} = 1$ and the $d_l$ are $p$-integral.  Write 
\[ h(\Lambda) = \sum_{i=1}^r h_i\Lambda^{e_i} \]
where $e_i\in E$ for $i=1,\dots,r$.  The coefficient of $\Lambda^l$ in the quotient $h(\Lambda)/D(\Lambda)^k$ is 
\begin{equation}
\sum_{i=1}^r h_i d_{l-e_i},
\end{equation}
with the understanding that $d_{l-e_i} = 0$ if $l-e_i\not\in E$.  Since the $d_{l-e_i}$ are $p$-integral
\[ \bigg\lvert \sum_{i=1}^r h_i d_{l-e_i}\bigg\rvert \leq \sup_{i=1,\dots,r} \lvert h_i\rvert, \]
which implies that $\lvert h(\Lambda)/D(\Lambda)^k\rvert\leq \lvert h(\Lambda)\rvert$.  

We must have $\lvert h(\Lambda)\rvert = \lvert h_i\rvert$ for some $i$.  To fix ideas, suppose that
\begin{equation}
\lvert h(\Lambda)\rvert = \lvert h_i\rvert \quad\text{for $i=1,\dots,s$} 
\end{equation}
and
\begin{equation}
\lvert h(\Lambda)\rvert > \lvert h_i\rvert \quad\text{for $i=s+1,\dots,r$}. 
\end{equation}
Since the cone ${\mathcal C}$ has a vertex at the origin, we can choose a vector $v\in{\mathbb R}^N$ such that the inner product $v\bullet w$ is $\geq 0$ for all $w\in{\mathcal C}$ and $v\bullet w=0$ only if $w={\bf 0}$.  For some $i\in\{1,\dots,s\}$, the inner product $v\bullet e_i$ is minimal.  To fix ideas, suppose that
\begin{equation}
v\bullet e_1\leq v\bullet e_i\quad\text{for $i=2,\dots,s$.}
\end{equation}
From (2.2), the coefficient of $\Lambda^{e_1}$ in $h(\Lambda)/D(\Lambda)^k$ is (using $d_{\bf 0} = 1$)
\begin{equation}
h_1 + \sum_{i=2}^r h_i d_{e_1-e_i}. 
\end{equation}

For $i=2,\dots,s$ we have by (2.5) that $v\bullet(e_1-e_i)\leq 0$.  But $e_1-e_i\neq{\bf 0}$, so $e_1-e_i\not\in E$.  This implies $d_{e_1-e_i} = 0$, so (2.6) simplifies to
\[ h_1 + \sum_{i=s+1}^r h_i d_{e_1-e_i}. \]
Equations (2.3) and (2.4) now imply that
\[ \bigg\lvert h_1 + \sum_{i=s+1}^r h_i d_{e_1-e_i}\bigg\rvert = \lvert h(\Lambda)\rvert, \]
hence $\lvert h(\Lambda)/D(\Lambda)^k\rvert\geq \lvert h(\Lambda)\rvert$.  
\end{proof}

It follows from the definition of $E$ that if $(l_1,\dots,l_N)\in E$, then $l_i\geq 0$ for $i=M+1,\dots,N$.  This implies that if $h(\Lambda)\in L_E$ and $\lambda\in{\mathcal D}$, then
\begin{equation}
\lvert h(\lambda)\rvert \leq \lvert h(\Lambda)\rvert.
\end{equation}
\begin{lemma}
Let $h_k(\Lambda)\in L_E$ for $k=1,2,\dots$ and suppose that the sequence $\{h_k(\Lambda)/D(\Lambda)^k\}_{k=1}^\infty$
converges to an element $\xi(\Lambda)\in R_E'$.  Then for all $\lambda\in{\mathcal D}$, the sequence $\{h_k(\lambda)/D(\lambda)^k\}_{k=1}^\infty$ converges in ${\mathbb C}_p$.
\end{lemma}

\begin{proof}
For $k<k'$ we have
\begin{align*}
\bigg\lvert \frac{h_k(\lambda)}{D^k(\lambda)} - \frac{h_{k'}(\lambda)}{D^{k'}(\lambda)}\bigg\rvert &= \bigg\lvert \frac{h_k(\lambda)D(\lambda)^{k'-k}-h_{k'}(\lambda)}{D^{k'}(\lambda)}\bigg\rvert \\
 & = \lvert h_k(\lambda)D(\lambda)^{k'-k}-h_{k'}(\lambda)\rvert \\
 & \leq \lvert h_k(\Lambda)D(\Lambda)^{k'-k}-h_{k'}(\Lambda)\rvert \\
 & = \bigg\lvert \frac{h_k(\Lambda)D(\Lambda)^{k'-k}-h_{k'}(\Lambda)}{D^{k'}(\Lambda)}\bigg\rvert \\
 & = \bigg\lvert \frac{h_k(\Lambda)}{D^k(\Lambda)} - \frac{h_{k'}(\Lambda)}{D^{k'}(\Lambda)}\bigg\rvert,
\end{align*}
where the second line follows from the requirement that $\lvert D(\lambda)\rvert = 1$ for $\lambda\in{\mathcal D}$, the third line holds by~(2.7), and the fourth line holds by Lemma~2.1.

Since the sequence $\{h_k(\Lambda)/D(\Lambda)^k\}_{k=1}^\infty$ converges, it is a Cauchy sequence.  This inequality shows that the sequence $\{h_k(\lambda)/D(\lambda)^k\}_{k=1}^\infty$ is also a Cauchy sequence, which must converge because ${\mathbb C}_p$ is complete.
\end{proof}

It is easy to see that the limit of the sequence $\{h_k(\lambda)/D(\lambda)^k\}_{k=1}^\infty$ is independent of the choice of sequence $\{h_k(\Lambda)/D(\Lambda)^k\}_{k=1}^\infty$ converging to $\xi(\Lambda)$.  We may therefore define
\[ \xi(\lambda) = \lim_{k\to\infty}  \frac{h_k(\lambda)}{D^k(\lambda)} . \]
Using this definition, the elements of $R_E'$ become functions on ${\mathcal D}$.  Note that the proof of Lemma~2.8 shows that, as functions on ${\mathcal D}$, the sequence $\{h_k/D^k\}_{k=1}^\infty$ converges uniformly to $\xi$.

We shall need the following fact later.
\begin{lemma}
If $\xi(\Lambda)\in R_E'$ and $\lambda\in{\mathcal D}$, then $\lvert \xi(\lambda)\vert\leq \lvert\xi(\Lambda)\rvert$.
\end{lemma}

\begin{proof}
It follows from Lemma 2.1, Eq.~(2.7), and the fact that $\lvert D(\lambda)\rvert = 1$ for $\lambda\in{\mathcal D}$ that
\begin{equation}
\bigg\lvert \frac{h(\lambda)}{D(\lambda)^k}\bigg\rvert\leq \bigg\lvert \frac{h(\Lambda)}{D(\Lambda)^k}\bigg\rvert
\end{equation}
for $h(\Lambda)\in L_E$, $k\in{\mathbb N}$, and $\lambda\in{\mathcal D}$.  If we choose a sequence $\{h_k(\Lambda)/D(\Lambda)^k\}_{k=1}^\infty$, with the $h_k(\Lambda)$ in $L_E$, converging to $\xi(\Lambda)\in R_E'$, then for $k$ sufficiently large we have 
\[ \lvert\xi(\Lambda)\rvert = \bigg\lvert\frac{h_k(\Lambda)}{D(\Lambda)^k}\bigg\rvert \]
and for $\lambda\in{\mathcal D}$ we have by the definition of $\xi(\lambda)$
\[ \lvert\xi(\lambda)\rvert = \bigg\lvert\frac{h_k(\lambda)}{D(\lambda)^k}\bigg\rvert. \]
The assertion of the lemma now follows from (2.10).
\end{proof}

The following result is our more precise version of Theorem 1.10.
\begin{theorem}
The entries in the matrix ${\mathcal F}(\Lambda)$ lie in $R_E'$.  
Let $\lambda\in ({\mathbb F}_q^\times)^M\times{\mathbb F}_q^{N-M}$ and let $\hat{\lambda}\in{\mathbb Q}_p(\zeta_{q-1})^N$ be its Teichm\"uller lifting.  If $\bar{D}(\lambda)\neq 0$, then $\hat{\lambda}^{p^i}\in{\mathcal D}$ for $i=0,\dots,a-1$ and 
\[ \rho(\lambda,t) = \det\big( I-t{\mathcal F}(\hat{\lambda}^{p^{a-1}}){\mathcal F}(\hat{\lambda}^{p^{a-1}})\cdots{\mathcal F}(\hat{\lambda})\big). \]
\end{theorem}
Most of this paper is devoted to proving the first sentence of Theorem 2.11.  The remaining assertion will then follow by an application of Dwork's $p$-adic cohomology theory.  

Unfortunately the rings $R_E$ and $R'_E$ are not large enough to contain all the relevant series we shall encounter and therefore need to be enlarged.  Let $\tilde{R} = R_E[\Lambda_1^{\pm 1},\dots,\Lambda_N^{\pm1}]$ and let 
\[ \tilde{R}' = R'_E[\Lambda_1^{\pm 1},\dots,\Lambda_M^{\pm 1},\Lambda_{M+1},\dots,\Lambda_N]. \]
It is clear from the definition of ${\mathcal D}$ that the elements of $\tilde{R}'$ are functions on ${\mathcal D}$.  
We define ${R}$ (resp.\ ${R}'$) to be the completion of~$\tilde{R}$ (resp.~$\tilde{R}'$) under the norm on ${\mathbb B}$.  Note that the elements of ${R}'$ define functions on~${\mathcal D}$.  It follows from Lemma~2.9 that 
\begin{equation}
\lvert\xi(\lambda)\rvert\leq \lvert\xi(\Lambda)\rvert\quad\text{for $\xi(\Lambda)\in {R}'$ and $\lambda\in{\mathcal D}$.}
\end{equation}

\begin{comment}
Set
\[ {\mathcal E} = \{ l=(l_1,\dots,l_N)\in{\mathbb Z}^N\mid  \text{$l_k\geq 0$ for $k=M+1,\dots,N$}\}. \]
Note that $E\subseteq{\mathcal E}$.
Let $\tilde{R}$ (resp.\ $\tilde{R}'$) be the ${\mathbb C}_p$-subspace of ${\mathbb B}$ spanned by the set $\bigcup_{l\in {\mathbb Z}^N} \Lambda^l R_E$ (resp.\ $\bigcup_{l\in {\mathcal E}} \Lambda^l R_E'$).  The ${\mathbb C}_p$-vector spaces $\tilde{R}$ and $\tilde{R}'$ are rings.  For example, the elements of $\tilde{R}$ are finite sums of expressions of the form $\Lambda^l \xi(\Lambda)$ with $l\in{\mathbb Z}^N$ and $\xi(\Lambda)\in R_E$.  For the product of two such expressions we have
\[ \big(\Lambda^l \xi(\Lambda)\big) \big(\Lambda^{l'} \xi'(\Lambda)\big) = \Lambda^{l+l'} \big(\xi(\Lambda)\xi'(\Lambda)\big). \]
Note that for $l\in{\mathcal E}$, the monomial $\Lambda^l$ is a well-defined function on ${\mathcal D}$, so the elements of $\tilde{R}'$ are functions on ${\mathcal D}$.
\end{comment}

{\bf Remark.}  One can show using the argument of Dwork\cite[Lemma~1.2]{D} that for $\xi(\Lambda)\in R'$ 
\[ \sup_{\lambda\in{\mathcal D}} \lvert \xi(\lambda)\rvert = \lvert \xi(\Lambda)\rvert. \]

Let ${\mathcal M}\subseteq{\mathbb Z}^N$ be the abelian group generated by $A$.  We define an ${\mathcal M}$-grading on ${R}$ and ${R}'$.  Every $\xi(\Lambda)\in {R}$ can be written as a series $\xi(\Lambda) = \sum_{l\in{\mathbb Z}^N} c_l\Lambda^l$.
We say that {\it $\xi(\Lambda)$ has degree $u\in{\mathcal M}$\/} if $\sum_{k=1}^N l_k{\bf a}_k = u$ for all $l\in{\mathbb Z}^N$ for which $c_l\neq 0$.  We denote by ${R}_u$ the ${\mathbb C}_p$-subspace of ${R}$ consisting of all elements of degree~$u$.  Each ${R}_u$ is complete in the norm and is a module over ${R}_{\bf 0}$, which is a ring.  Identical remarks apply to the induced grading on ${R}'$.  

\section{$p$-adic estimates}

In this section we recall some estimates from \cite[Section 3]{AS1} that will be applied later.
Let $ {\rm AH}(t)= \exp(\sum_{i=0}^{\infty}t^{p^i}/p^i)$ be the Artin-Hasse series, a power series in $t$ that has $p$-integral coefficients, let $\gamma_0$ be a zero of the series $\sum_{i=0}^\infty t^{p^i}/p^i$ having ${\rm ord}\;\gamma_0 = 1/(p-1)$, and set 
\[ \theta(t) = {\rm AH}({\gamma}_0t)=\sum_{i=0}^{\infty}\theta_i t^i. \]
We then have
\begin{equation}
{\rm ord}\: \theta_i\geq \frac{i}{p-1}.
\end{equation}

We define $\hat{\theta}(t) = \prod_{j=0}^\infty \theta(t^{p^j})$, which gives $\theta(t) = \hat{\theta}(t)/\hat{\theta}(t^p)$.   If we set 
\begin{equation}
\gamma_j = \sum_{i=0}^j \frac{\gamma_0^{p^i}}{p^i},
\end{equation}
then
\begin{equation}
\hat{\theta}(t) = \exp\bigg(\sum_{j=0}^{\infty} \gamma_j t^{p^j}\bigg) = \prod_{j=0}^\infty \exp(\gamma_j t^{p^j}).
\end{equation}
If we write $\hat{\theta}(t) = \sum_{i=0}^\infty \hat{\theta}_i(\gamma_0 t)^i/i!$, then by \cite[Eq.~(3.8)]{AS1} we have
\begin{equation}
{\rm ord}\:\hat{\theta}_i\geq 0.
\end{equation}

We shall also need the series
\begin{equation}
\hat{\theta}_1(t) := \prod_{j=1}^\infty \exp(\gamma_jt^{p^j}) = :\sum_{i=0}^\infty \frac{\hat{\theta}_{1,i}}{i!}(\gamma_0 t)^i. 
\end{equation}
Note that $\hat{\theta}(t) = \exp(\gamma_0t)\hat{\theta}_1(t)$.  
By \cite[Eq.~(3.10)]{AS1}
\begin{equation}
{\rm ord}\:\hat{\theta}_{1,i}\geq \frac{i(p-1)}{p}. 
\end{equation}

Define the series $\hat{\theta}_1(\Lambda,x)$ by the formula
\begin{equation}
\hat{\theta}_1(\Lambda,x) = \prod_{k=1}^N \hat{\theta}_1(\Lambda_kx^{{\bf a}_k}).
\end{equation}
Expanding the product~(3.7) according to powers of $x$ we get
\begin{equation}
\hat{\theta}_1(\Lambda,x) = \sum_{u=(u_0,\dots,u_{n+1})\in{\mathbb N}A} \hat{\theta}_{1,u}(\Lambda)\gamma_0^{u_{n+1}} x^u,
\end{equation}
where
\begin{equation}
\hat{\theta}_{1,u}(\Lambda) = \sum_{\substack{k_1,\dots,k_N\in{\mathbb N}\\ \sum_{j=1}^N k_j{\bf a}_j = u}} \bigg(\prod_{j=1}^N \frac{\hat{\theta}_{1,k_j}}{k_j!}\bigg)\Lambda_1^{k_1}\cdots\Lambda_N^{k_N}.
\end{equation}

We have similar results for the reciprocal power series
\[ \hat{\theta}_1(t)^{-1} = \prod_{j=1}^\infty \exp(-\gamma_jt^{p^j}). \]
If we write
\begin{equation}
\hat{\theta}_1(t)^{-1} = \sum_{i=0}^\infty \frac{\hat{\theta}'_{1,i}}{i!}(\gamma_0 t)^i,
\end{equation}
then the coefficients satisfy
\begin{equation}
{\rm ord}\: \hat{\theta}'_{1,i}\geq \frac{i(p-1)}{p}.
\end{equation}
We also have
\begin{equation}
\hat{\theta}_1(\Lambda,x)^{-1} = \prod_{k=1}^N \hat{\theta}_1(\Lambda_kx^{{\bf a}_k})^{-1},
\end{equation}
which we again expand in powers of $x$ as
\begin{equation}
\hat{\theta}_1(\Lambda,x)^{-1} = \sum_{u=(u_0,\dots,u_{n+1})\in{\mathbb N}A} \hat{\theta}'_{1,u}(\Lambda)\gamma_0^{u_{n+1}} x^u
\end{equation}
with
\begin{equation}
\hat{\theta}'_{1,u}(\Lambda) = \sum_{\substack{k_1,\dots,k_N\in{\mathbb N}\\ \sum_{j=1}^N k_j{\bf a}_j = u}} \bigg(\prod_{j=1}^N \frac{\hat{\theta}'_{1,k_j}}{k_j!}\bigg) \Lambda_1^{k_1}\cdots\Lambda_N^{k_N}.
\end{equation}

We also define
\begin{equation}
\theta(\Lambda,x) = \prod_{k=1}^N \theta(\Lambda_kx^{{\bf a}_k}).
\end{equation}
Expanding the right-hand side in powers of $x$, we have 
\begin{equation}
\theta(\Lambda,x) = \sum_{u\in{\mathbb N}A} \theta_u(\Lambda)x^u,
\end{equation}
where
\begin{equation}
\theta_u(\Lambda) = \sum_{\nu\in{\mathbb N}^N} \theta^{(u)}_\nu\Lambda^\nu
\end{equation}
and
\begin{equation}
\theta^{(u)}_\nu = \begin{cases} \prod_{k=1}^N \theta_{\nu_k} & \text{if $\sum_{k=1}^N \nu_k{\bf a}_k = u$,} \\
0 & \text{if $\sum_{k=1}^N \nu_k{\bf a}_k \neq u$,} \end{cases}
\end{equation}
so $\theta_u(\Lambda)$ is homogeneous of degree $u$.
The equation $\sum_{k=1}^N \nu_k{\bf a}_k = u$ has only finitely many solutions $\nu\in{\mathbb N}^N$, so $\theta_u(\Lambda)$ is a polynomial in the $\Lambda_k$.  Equations (3.1) and (3.18) show that
\begin{equation}
{\rm ord}_p\:\theta^{(u)}_\nu \geq\frac{\sum_{j=1}^N \nu_j}{p-1} = \frac{u_{n+1}}{p-1}.
\end{equation}

\section{The Dwork-Frobenius operator}

We define the spaces $S$ and $S'$ and Dwork's Frobenius operator on those spaces.

Note that the abelian group ${\mathcal M}$ generated by $A$ lies in the hyperplane $\sum_{i=0}^n u_i = du_{n+1}$ in ${\mathbb R}^{n+2}$.  Set ${\mathcal M}_- = {\mathcal M}\cap({\mathbb Z}_{<0})^{n+2}$.  We denote by $\delta_-$ the truncation operator on formal Laurent series in variables $x_0,\dots,x_{n+1}$ that preserves only those terms having all exponents negative:
\[ \delta_-\bigg(\sum_{k\in{\mathbb Z}^{n+2}}c_kx^k\bigg) = \sum_{k\in({\mathbb Z}_{<0})^{n+2}} c_kx^k. \]
We use the same notation for formal Laurent series in a single variable $t$:
\[ \delta_-\bigg(\sum_{k=-\infty}^\infty c_kt^k\bigg) = \sum_{k=-\infty}^{-1} c_kt^k. \]
It is straightforward to check that if $\xi_1$ and $\xi_2$ are two series for which the product $\xi_1\xi_2$ is defined and if no monomial in $\xi_2$ has a negative exponent, then
\begin{equation}
\delta_-\big(\delta_-(\xi_1)\xi_2\big) = \delta_-(\xi_1\xi_2). 
\end{equation} 

Define $S$ to be the ${\mathbb C}_p$-vector space of formal series
\[ S = \bigg\{\xi(\Lambda,x) = \sum_{u\in {\mathcal M}_-} \xi_{u}(\Lambda) \gamma_0^{u_{n+1}}x^{u} \mid 
\text{$\xi_{u}(\Lambda)\in {R}_{u}$ and $\{\lvert\xi_{u}\rvert\}_u$ is bounded}\bigg\}. \]
Let $S'$ be defined analogously with the condition ``$\xi_{u}(\Lambda)\in {R}_{u}$'' being replaced by 
``$\xi_{u}(\Lambda)\in {R}_{u}'$''.  Define a norm on $S$ by setting
\[ \lvert\xi(\Lambda,x)\rvert = \sup_{u\in{\mathcal M}_-}\{\lvert\xi_{u}\rvert\}. \]
Both $S$ and $S'$ are complete under this norm.  

Let 
\[ \xi(\Lambda,x) = \sum_{\nu\in {\mathcal M}_-} \xi_\nu(\Lambda) \gamma_0^{\nu_{n+1}}x^\nu\in S. \]
We show that the product $\theta(\Lambda,x) \xi(\Lambda^p,x^p)$ is well-defined as a formal series in $x$.  We have formally
\[ \theta(\Lambda,x) \xi(\Lambda^p,x^p) = \sum_{\rho\in{\mathcal M}} \zeta_\rho(\Lambda)x^\rho, \]
where
\begin{equation}
\zeta_\rho(\Lambda) = \sum_{\substack{u\in{\mathbb N}A,\nu\in{\mathcal M}_-\\ u+p\nu = \rho}} \gamma_0^{\nu_{n+1}}\theta_u(\Lambda)\xi_\nu(\Lambda^p).
\end{equation}
Since $\theta_u(\Lambda)$ is a polynomial, the product $\theta_u(\Lambda)\xi_\nu(\Lambda^p)$ is a well-defined element of~${R}_{\rho}$.  It follows from (3.17), (3.19), and the equality $u+p\nu=\rho$ that the coefficients of $\gamma_0^{\nu_{n+1}}\theta_u(\Lambda)$ all have $p$-ordinal at least $\big(\rho_{n+1}/(p-1)\big) - \nu_{n+1}$.  Since $\lvert\xi_\nu(\Lambda)\rvert$ is bounded independently of $\nu$ and there are only finitely many terms on the right-hand side of (4.2) with a given value of $\nu_{n+1}$, the series (4.2) converges to an element of ${R}_\rho$ (because $-\nu_{n+1}\to\infty$ as $\nu\to\infty$).  This estimate also shows that if $\xi(\Lambda,x)\in S'$, then $\zeta_\rho(\Lambda)\in {R}'_\rho$.  

Define for $\xi(\Lambda,x)\in S$
\begin{align*}
\alpha^*\big(\xi(\Lambda,x)\big) &= \delta_-\big(\theta(\Lambda,x)\xi(\Lambda^p,x^p)\big) \\
 &= \sum_{\rho\in {\mathcal M}_-} \zeta_{\rho}(\Lambda)x^{\rho}.
\end{align*}
For $\rho\in {\mathcal M}_-$, put $\eta_{\rho}(\Lambda) = \gamma_0^{-\rho_{n+1}}\zeta_{\rho}(\Lambda)$, so that
\begin{equation}
\alpha^*\big(\xi(\Lambda,x)\big) = \sum_{\rho\in {\mathcal M}_-} \eta_{\rho}(\Lambda)\gamma_0^{\rho_{n+1}}x^{\rho}
\end{equation}
with (by (4.2))
\begin{equation}
\eta_{\rho}(\Lambda) = \sum_{\substack{u\in{\mathbb N}A,\,\nu\in {\mathcal M}_-\\ u+p\nu=\rho}}\gamma_0^{-\rho_{n+1}+\nu_{n+1}}\theta_u(\Lambda)\xi_{\nu}(\Lambda^p).
\end{equation}

\begin{proposition}
The map $\alpha^*$ is an endomorphism of $S$ and of $S'$, and for $\xi(\Lambda,x)\in S$ we have
\begin{equation}
\lvert\alpha^*\big(\xi(\Lambda,x)\big)\rvert\leq \lvert p\xi(\Lambda,x)\rvert.
\end{equation}
\end{proposition}

\begin{proof}
By (4.3), the proposition will follow from the estimate
\[ \lvert\eta_{\rho}(\Lambda)\rvert\leq \lvert p\xi(\Lambda,x)\rvert\quad\text{for all $\rho\in {\mathcal M}_-$.} \]
Using (4.4), we see that this estimate will follow in turn from the estimate
\[ \lvert\gamma_0^{-\rho_{n+1}+\nu_{n+1}}\theta_u(\Lambda)\rvert\leq \lvert p\rvert \]
for all $u\in{\mathbb N}A$, $\nu\in {\mathcal M}_-$, with $u+p\nu = \rho$.  From (3.17) and (3.19) we see that all coefficients of $\gamma_0^{-\rho_{n+1}+\nu_{n+1}}\theta_u(\Lambda)$ have $p$-ordinal greater than or equal to
\[ \frac{-\rho_{n+1}+\nu_{n+1}+u_{n+1}}{p-1}. \]
Since $u+p\nu=\rho$, this expression simplifies to $-\nu_{n+1}$, so
\begin{equation}
\lvert\gamma_0^{-\rho_{n+1}+\nu_{n+1}}\theta_u(\Lambda)\rvert\leq \lvert p\rvert^{-\nu_{n+1}}
\end{equation}
and $-\nu_{n+1}\geq 1$ since $\nu\in {\mathcal M}_-$.  
\end{proof}

Note that the equality $-\nu_{n+1} = 1$ occurs for only $M$ points $\nu\in{\mathcal M}_-$, namely,
$\nu = -{\bf a}_1,\dots,-{\bf a}_M$.  The following corollary is then an immediate consequence of the proof of Proposition~4.5.
\begin{corollary}
If $\xi_{\nu}(\Lambda) = 0$ for $\nu = -{\bf a}_1,\dots,-{\bf a}_M$, then $\lvert\alpha^*(\xi(\Lambda,x))\rvert\leq \lvert p^{2}\xi(\Lambda,x)\rvert$.
\end{corollary}

\section{A technical lemma}

The action of $\alpha^*$ is extended to $S^M$ componentwise: if
\begin{equation}
 \xi(\Lambda,x) = (\xi^{(1)}(\Lambda,x),\dots,\xi^{(M)}(\Lambda,x))\in S^M, 
\end{equation}
then
\begin{equation}
 \alpha^*\big(\xi(\Lambda,x)\big) = \big(\alpha^*\big(\xi^{(1)}(\Lambda,x)\big),\dots,\alpha^*\big(\xi^{(M)}(\Lambda,x)\big)\big). 
\end{equation}
Let $\xi(\Lambda,x)$ be as in (5.1) with 
\begin{equation}
\xi^{(i)}(\Lambda,x) =  \sum_{\nu\in {\mathcal M}_-} \xi_\nu^{(i)}(\Lambda) \gamma_0^{\nu_{n+1}}x^\nu.
\end{equation}
Since $\xi^{(i)}_{-{\bf a}_j}(\Lambda)\in R_{-{\bf a}_j}$, the matrix
\[ {\mathbb M}\big(\xi(\Lambda,x)\big):= \big( \Lambda_j\xi^{(i)}_{-{\bf a}_j}(\Lambda)\big)_{i,j=1}^M \]
has entries in $R_{\bf 0}$.  If $\xi(\Lambda,x)\in (S')^M$, then ${\mathbb M}\big(\xi(\Lambda,x)\big)$ has entries in $R'_{\bf 0}$.  The map $\xi(\Lambda,x)\mapsto {\mathbb M}(\xi(\Lambda,x))$ is a ${\mathbb C}_p$-linear map from $S^M$ to the ${\mathbb C}_p$-vector space of $(M\times M)$-matrices with entries in $R_{\bf 0}$.  Note that if $Y(\Lambda)$ is an $(M\times M)$-matrix with entries in $R_{\bf 0}$, then
\begin{equation}
{\mathbb M}\big(Y(\Lambda)\xi(\Lambda,x)\big) = Y(\Lambda) {\mathbb M}\big(\xi(\Lambda,x)\big), 
\end{equation}
where we regard $\xi(\Lambda,x)$ as a column vector for the purpose of matrix multiplication.
In particular, if ${\mathbb M}\big(\xi(\Lambda,x)\big)$ is an invertible matrix, then
\begin{equation}
 {\mathbb M}\bigg( {\mathbb M}\big(\xi(\Lambda,x)\big)^{-1}\xi(\Lambda,x)\bigg) = I. 
 \end{equation}

We extend the norm on $S$ to $S^M$ and the norm on $R$ to ${\rm Mat}_M(R)$, the $(M\times M)$-matrices with entries in $R$.  For $\xi(\Lambda,x)$ as in (5.1), we define
\[ \lvert \xi(\Lambda,x) \rvert = \max \{\lvert\xi^{(i)}(\Lambda,x) \rvert \}_{i=1}^M \]
and for $Y(\Lambda) = \big(Y_{ij}(\Lambda)\big)_{i,j=1}^M$ with $Y_{ij}(\Lambda)\in R$ we define
\[ \lvert Y(\Lambda)\rvert = \max\{\lvert Y_{ij}(\Lambda)\rvert\}_{i,j=1}^M. \]

Note that the matrix ${\mathbb M}\big(\xi(\Lambda,x)\big)$ has an inverse with entries in $R_{\bf 0}$ if and only if $\det 
{\mathbb M}\big(\xi(\Lambda,x)\big)$ is an invertible element of $R_{\bf 0}$.  Likewise, if $\xi(\Lambda,x)\in (S')^M$, this matrix has an inverse with entries in $R'_{\bf 0}$ if and only if its determinant is an invertible element of $R'_{\bf 0}$.  The main result of this section is the following assertion.
\begin{lemma}
Let $\xi(\Lambda,x)\in S^M$ with ${\mathbb M}\big(\xi(\Lambda,x)\big)$ invertible, $\big\lvert {\mathbb M}\big(\xi(\Lambda,x)\big)\big\rvert = \lvert\xi(\Lambda,x)\rvert$, and $\big\lvert\det{\mathbb M}\big(\xi(\Lambda,x)\big)\big\rvert  = \lvert\xi(\Lambda,x)\rvert^M$.  Then ${\mathbb M}\big(\alpha^*(\xi(\Lambda,x))\big)$ is invertible, 
\begin{equation}
 \big\lvert {\mathbb M}\big(\alpha^*(\xi(\Lambda,x))\big)\big\rvert = \big\lvert\alpha^*\big(\xi(\Lambda,x)\big)\big\rvert = \lvert p\xi(\Lambda,x)\rvert,
\end{equation}
and
\begin{equation}
 \big\lvert\det {\mathbb M}\big(\alpha^*(\xi(\Lambda,x))\big)\big\rvert = \lvert p\xi(\Lambda,x)\rvert^M. 
\end{equation}
\end{lemma}

{\bf Remark.}  To prove Lemma 5.6, it suffices to show that ${\mathbb M}\big(\alpha^*(\xi(\Lambda,x))\big)$ is invertible and that (5.8) holds: from Proposition 4.5 we get
\[ \big\lvert {\mathbb M}\big(\alpha^*(\xi(\Lambda,x))\big)\big\rvert \leq  \big\lvert\alpha^*\big(\xi(\Lambda,x)\big)\big\rvert \leq \lvert p\xi(\Lambda,x)\rvert. \]
If any of these inequalities were strict, the usual formula for the determinant of a matrix in terms of its entries would imply that
\[ \big\lvert\det {\mathbb M}\big(\alpha^*(\xi(\Lambda,x))\big)\big\rvert < \lvert p\xi(\Lambda,x)\rvert^M. \]
Similar reasoning applies to the inverse of ${\mathbb M}\big(\alpha^*(\xi(\Lambda,x))\big)$.  
By (5.8) we have
\begin{equation}
\big\lvert\det {\mathbb M}\big(\alpha^*(\xi(\Lambda,x))\big)^{-1}\big\rvert = \lvert p\xi(\Lambda,x)\rvert^{-M}.
\end{equation}
The usual formula for the inverse of a matrix in terms of its cofactors implies that 
\[ \big\lvert{\mathbb M}\big(\alpha^*(\xi(\Lambda,x))\big)^{-1}\big\rvert\leq \frac{1}{\lvert p\xi(\Lambda,x)\rvert}. \]
But if this inequality were strict it would lead to a contradiction of (5.9), so we have the following corollary.
\begin{corollary}
Under the hypotheses of Lemma 5.6, 
\[ \big\lvert {\mathbb M}\big(\alpha^*(\xi(\Lambda,x))\big)^{-1}\big\rvert=\frac{1}{\lvert p\xi(\Lambda,x)\rvert}. \]
\end{corollary}

The proof of Lemma 5.6 will require several steps.  We first consider a matrix constructed from $\theta(\Lambda,x)$ (Eq.~(3.16)):
\[ {\mathbb M}_\theta(\Lambda) = \big( \theta_{p{\bf a}_i-{\bf a}_j}(\Lambda)\big)_{i,j=1}^M. \]
From (3.17) and (3.18) we have explicit formulas for these matrix entries:
\[ \theta_{p{\bf a}_i-{\bf a}_j}(\Lambda) = \sum_{\substack{\nu\in {\mathbb N}^N\\ \sum_{k=1}^N \nu_k{\bf a}_k = p{\bf a}_i-{\bf a}_j}} \bigg(\prod_{k=1}^N \theta_{\nu_k}\bigg) \Lambda^\nu. \]
Furthermore, since the last coordinate of $p{\bf a}_i-{\bf a}_j$ equals $p-1$, each $\nu_k$ in this summation is $\leq p-1$, so from their definition $\theta_{\nu_k} = \gamma_0^{\nu_k}/\nu_k!$.  Since $\sum_{k=1}^N \nu_k = p-1$, we have
\[ \theta_{p{\bf a}_i-{\bf a}_j}(\Lambda) = \sum_{\substack{\nu\in {\mathbb N}^N\\ \sum_{k=1}^N \nu_k{\bf a}_k = p{\bf a}_i-{\bf a}_j}} \frac{\gamma_0^{p-1} \Lambda^\nu}{\nu_1!\cdots\nu_N!}. \]
Thus 
\begin{equation}
{\mathbb M}_\theta(\Lambda) = \gamma_0^{p-1}H(\Lambda),
\end{equation}
 where $H(\Lambda)$ is given in Eq.~(1.4), and
\begin{equation}
C(\Lambda)^{-p} {\mathbb M}_\theta(\Lambda) C(\Lambda) = \gamma_0^{p-1}B(\Lambda),
\end{equation}
where $B(\Lambda)$ is given in Eq.~(1.5).  It now follows from the discussion in Section~1 that $C(\Lambda)^{-p}{\mathbb M}_\theta(\Lambda) C(\Lambda)$ is an invertible element of~$R_{\bf 0}'$.  Furthermore, we have
\begin{equation}
\det C(\Lambda)^{-p} {\mathbb M}_\theta(\Lambda) C(\Lambda) = \gamma_0^{(p-1)M} D(\Lambda).  
\end{equation}
And since $\lvert D(\Lambda)\rvert = 1$ (since $D(\Lambda)$ has constant term 1 and integral coefficients) we get
\begin{equation}
\lvert \det C(\Lambda)^{-p} {\mathbb M}_\theta(\Lambda) C(\Lambda) \rvert = \lvert p^M\rvert.
\end{equation}

With $\xi(\Lambda,x)$ as in (5.1), we rewrite (5.2) as
\[ \alpha^*\big(\xi(\Lambda,x)\big) = \big( \eta^{(1)}(\Lambda,x),\dots,\eta^{(M)}(\Lambda,x)\big), \]
where (see (4.2) and (4.3))
\begin{equation}
\eta^{(i)}(\Lambda,x) = \sum_{\rho\in {\mathcal M}_-} \eta^{(i)}_{\rho}(\Lambda)\gamma_0^{\rho_{n+1}}x^{\rho}
\end{equation}
with
\begin{equation}
\eta^{(i)}_{\rho}(\Lambda) = \sum_{\substack{u\in{\mathbb N}A,\,\nu\in {\mathcal M}_-\\ u+p\nu=\rho}}\gamma_0^{-\rho_{n+1}+\nu_{n+1}}\theta_u(\Lambda)\xi^{(i)}_{\nu}(\Lambda^p).
\end{equation}
We then have
\begin{equation}
{\mathbb M}\big(\alpha^*(\xi(\Lambda,x))\big) = \big( \Lambda_j\eta^{(i)}_{-{\bf a}_j}(\Lambda)\big)_{i,j=1}^M.
\end{equation}

If $\nu\neq -{\bf a}_1,\dots,-{\bf a}_M$, then $\nu_{n+1}\leq -2$, so by (5.16) we may write
\begin{multline}
\eta^{(i)}_{-{\bf a}_j}(\Lambda) = \sum_{k=1}^M \theta_{p{\bf a}_k-{\bf a}_j}(\Lambda)\xi^{(i)}_{-{\bf a}_k}(\Lambda^p) \\
 + \sum_{\substack{u\in{\mathbb N}A,\,\nu\in {\mathcal M}_-\\ u+p\nu=-{\bf a}_j\\\nu_{n+1}\leq -2}}\gamma_0^{-\rho_{n+1}+\nu_{n+1}}\theta_u(\Lambda)\xi^{(i)}_{\nu}(\Lambda^p).
\end{multline}
We write ${\mathbb M}\big(\alpha^*(\xi(\Lambda,x))\big) = {\mathbb M}^{(1)}(\Lambda) + {\mathbb M}^{(2)}(\Lambda)$, where
\begin{equation}
{\mathbb M}^{(1)}_{ij}(\Lambda) = \Lambda_j\sum_{k=1}^M \theta_{p{\bf a}_k-{\bf a}_j}(\Lambda)\xi^{(i)}_{-{\bf a}_k}(\Lambda^p)
\end{equation}
and
\begin{equation}
{\mathbb M}^{(2)}_{ij}(\Lambda) = \Lambda_j\sum_{\substack{u\in{\mathbb N}A,\,\nu\in {\mathcal M}_-\\ u+p\nu=-{\bf a}_j\\\nu_{n+1}\leq -2}}\gamma_0^{-\rho_{n+1}+\nu_{n+1}}\theta_u(\Lambda)\xi^{(i)}_{\nu}(\Lambda^p).
\end{equation}
It follows from (5.19), (5.20), and a short calculation that
\begin{equation}
{\mathbb M}\big(\alpha^*(\xi(\Lambda,x))\big) = {\mathbb M}\big(\xi(\Lambda^p,x)\big)C(\Lambda)^{-p} {\mathbb M}_\theta(\Lambda)C(\Lambda) + {\mathbb M}^{(2)}(\Lambda).
\end{equation}

\begin{proof}[Proof of Lemma 5.6]
For notational convenience set 
\[ \widetilde{\mathbb M}(\Lambda) = {\mathbb M}\big(\xi(\Lambda^p,x)\big)C(\Lambda)^{-p} {\mathbb M}_\theta(\Lambda)C(\Lambda) \]
so that (5.21) can be rewritten as
\begin{equation}
{\mathbb M}\big(\alpha^*(\xi(\Lambda,x))\big) = \widetilde{\mathbb M}(\Lambda) + {\mathbb M}^{(2)}(\Lambda).
\end{equation}
From (5.12) we get
\[ \lvert \widetilde{\mathbb M}(\Lambda) \rvert \leq \lvert p{\mathbb M}(\xi(\Lambda^p,x))\rvert  = \lvert p{\mathbb M}(\xi(\Lambda,x))\rvert \]
and from (5.13)
\[ \lvert \det \widetilde{\mathbb M}(\Lambda) \rvert = \lvert p^M\det {\mathbb M}(\xi(\Lambda^p,x))\rvert = \lvert p^M\det {\mathbb M}(\xi(\Lambda,x))\rvert. \]
Applying our hypotheses gives
\begin{equation}
\lvert \widetilde{\mathbb M}(\Lambda) \rvert \leq \lvert p\xi(\Lambda,x)\rvert 
\end{equation}
and
\begin{equation}
 \lvert \det \widetilde{\mathbb M}(\Lambda) \rvert = \lvert p \xi(\Lambda,x)\rvert^M.
\end{equation}
And since the equality in (5.24) would fail if the inequality in (5.23) were strict, we must have
\begin{equation}
\lvert \widetilde{\mathbb M}(\Lambda) \rvert = \lvert p\xi(\Lambda,x)\rvert.
\end{equation}

The matrix $\widetilde{\mathbb M}(\Lambda)$ is invertible by (5.12) and our hypotheses, and (5.24) implies
\begin{equation}
\lvert \det\widetilde{\mathbb M}(\Lambda)^{-1}\rvert = \lvert p \xi(\Lambda,x)\rvert^{-M}.
\end{equation}
Arguing as in the derivation of Corollary 5.10 from (5.9) then gives
\begin{equation}
\lvert \widetilde{\mathbb M}(\Lambda)^{-1} \rvert = \frac{1}{\lvert p\xi(\Lambda,x)\rvert}.
\end{equation}

Rewrite (5.22) as
\begin{equation}
{\mathbb M}\big(\alpha^*(\xi(\Lambda,x))\big) = \widetilde{\mathbb M}(\Lambda)\big( I + \widetilde{\mathbb M}(\Lambda)^{-1}{\mathbb M}^{(2)}(\Lambda)\big).
\end{equation}
Estimate (4.7) implies that 
\begin{equation}
\lvert {\mathbb M}^{(2)}(\Lambda)\rvert\leq \lvert p^2\xi(\Lambda,x)\rvert,
\end{equation}
so by (5.27) 
\begin{equation}
\lvert  \widetilde{\mathbb M}(\Lambda)^{-1}{\mathbb M}^{(2)}(\Lambda)\rvert \leq\lvert p\rvert.
\end{equation}
It follows from (5.30) that $I + \widetilde{\mathbb M}(\Lambda)^{-1}{\mathbb M}^{(2)}(\Lambda)$ is invertible, so by (5.28) the invertibility of $\widetilde{\mathbb M}(\Lambda)$ implies the invertibility of 
${\mathbb M}\big(\alpha^*(\xi(\Lambda,x))\big)$.  Estimate (5.30) implies that
\begin{equation}
\big\lvert \det \big(I + \widetilde{\mathbb M}(\Lambda)^{-1}{\mathbb M}^{(2)}(\Lambda)\big) \big\rvert
 = 1,
\end{equation}
so (5.8) follows from (5.28) and (5.24).
\end{proof}

\section{Contraction mapping}

We use the Dwork-Frobenius operator to construct a contraction mapping on a subset of $S^M$.  Finding the fixed point of this contraction mapping will be the crucial step in proving the first assertion of Theorem~2.11. 

Put
\[ T = \{ \xi(\Lambda,x)\in S^M\mid \text{${\mathbb M}(\xi(\Lambda,x)) = I$ and $\lvert \xi(\Lambda,x)\rvert = 1$}\} \]
and put $T' = T\cap (S')^M$.  Note that $T$ and $T'$ are closed in the topology on $S^M$.  Elements of $T$ satisfy the hypotheses of Lemma~5.6.  By that result, if $\xi(\Lambda,x)\in T$, then ${\mathbb M}(\alpha^*(\xi(\Lambda,x)))$ is invertible, so we may define
\[ \phi\big(\xi(\Lambda,x)\big) := {\mathbb M}(\alpha^*(\xi(\Lambda,x)))^{-1} \alpha^*\big(\xi(\Lambda,x)\big), \]
where we regard $\alpha^*(\xi(\Lambda,x))$ on the right-hand side as a column vector for the purpose of matrix multiplication.  By Eq.~(5.5) we have
\begin{equation}
{\mathbb M}\big(\phi(\xi(\Lambda,x))\big) = I.
\end{equation}
Equation (4.6) and Corollary 5.10 imply that $\lvert \phi(\xi(\Lambda,x))\rvert\leq 1$, so in fact
\begin{equation}
\lvert \phi(\xi(\Lambda,x))\rvert = 1
\end{equation}
by (6.1).  Equations (6.1) and (6.2) show that $\phi(T)\subseteq T$ and $\phi(T')\subseteq T'$.  

\begin{proposition}
The operator $\phi$ is a contraction mapping on $T$.  More precisely, if $\xi^{(1)}(\Lambda,x),\xi^{(2)}(\Lambda,x)\in T$, then
\[ \big\lvert \phi\big(\xi^{(1)}(\Lambda,x)\big) - \phi\big(\xi^{(2)}(\Lambda,x)\big)\big\rvert \leq \lvert p\rvert \cdot\lvert \xi^{(1)}(\Lambda,x) - \xi^{(2)}(\Lambda,x)\rvert. \]
\end{proposition}

\begin{proof}
For notational convenience we sometimes write $\eta^{(i)}(\Lambda,x) = \alpha^*(\xi^{(i)}(\Lambda,x))$ for $i=1,2$.  Then
\begin{multline}
\phi\big(\xi^{(1)}(\Lambda,x)\big) - \phi\big(\xi^{(2)}(\Lambda,x)) = \\
{\mathbb M}\big(\eta^{(1)}(\Lambda,x)\big)^{-1} \eta^{(1)}(\Lambda,x)-{\mathbb M}\big(\eta^{(2)}(\Lambda,x)\big)^{-1} \eta^{(2)}(\Lambda,x) = \\
 {\mathbb M}\big(\eta^{(1)}(\Lambda,x)\big)^{-1}\bigg(\eta^{(1)}(\Lambda,x)-\eta^{(2)}(\Lambda,x)\bigg) \\
-{\mathbb M}\big(\eta^{(1)}(\Lambda,x)\big)^{-1}\bigg({\mathbb M}\big(\eta^{(1)}(\Lambda,x) - \eta^{(2)}(\Lambda,x)\big)\bigg) {\mathbb M}\big(\eta^{(2)}(\Lambda,x)\big)^{-1}\eta^{(2)}(\Lambda,x).
\end{multline}

The difference $\xi^{(1)}(\Lambda,x)-\xi^{(2)}(\Lambda,x)$ satisfies the hypothesis of Corollary 4.8, so
\begin{equation}
 \big\lvert \alpha^*\big(\xi^{(1)}(\Lambda,x)-\xi^{(2)}(\Lambda,x)\big)\big\rvert\leq \lvert p^2\rvert\cdot \lvert\xi^{(1)}(\Lambda,x)-\xi^{(2)}(\Lambda,x)\rvert. 
\end{equation}
By Corollary 5.10 we have
\begin{equation}
\lvert {\mathbb M}\big(\alpha^*(\xi^{(1)}(\Lambda,x))\big)^{-1} \rvert =  \lvert {\mathbb M}\big(\alpha^*(\xi^{(2)}(\Lambda,x))\big)^{-1} \rvert=1/\lvert p\rvert, 
\end{equation}
so
\begin{equation}
\bigg\lvert  {\mathbb M}\big(\eta^{(1)}(\Lambda,x)\big)^{-1}\bigg(\eta^{(1)}(\Lambda,x)-\eta^{(2)}(\Lambda,x)\bigg)\bigg\rvert
\leq \lvert p\rvert\cdot \lvert\xi^{(1)}(\Lambda,x)-\xi^{(2)}(\Lambda,x)\rvert.
\end{equation}
By Proposition 4.5 we have
\[ \big\lvert \alpha^*\big(\xi^{(2)}(\Lambda,x)\big)\big\rvert\leq \lvert p\rvert, \]
and by (6.5) we have 
\[ \big\lvert {\mathbb M}\big(\alpha^*\big(\xi^{(1)}(\Lambda,x) - \xi^{(2)}(\Lambda,x)\big)\big)\big\rvert\leq
\lvert p^2\rvert\cdot \lvert\xi^{(1)}(\Lambda,x)-\xi^{(2)}(\Lambda,x)\rvert, \]
so
\begin{multline}
\bigg\lvert{\mathbb M}\big(\eta^{(1)}(\Lambda,x)\big)^{-1}\bigg({\mathbb M}\big(\eta^{(1)}(\Lambda,x) - \eta^{(2)}(\Lambda,x)\big)\bigg) {\mathbb M}\big(\eta^{(2)}(\Lambda,x)\big)^{-1}\eta^{(2)}(\Lambda,x)\bigg\rvert \\
\leq \lvert p\rvert\cdot \lvert\xi^{(1)}(\Lambda,x)-\xi^{(2)}(\Lambda,x)\rvert.
\end{multline}
The assertion of the proposition  now follows from (6.4), (6.7), and (6.8).
\end{proof}

By a well-known theorem, the contraction mapping $\phi$ has a unique fixed point, and since $\phi$ is stable on $T'$ that fixed point lies in $T'$.  In the next section we discuss certain hypergeometric series and their $p$-adic relatives.  These $p$-adic relatives will be used to describe explicitly this fixed point.

\section{$A$-hypergeometric series}

The entries of the matrix $F(\Lambda)$ are $A$-hypergeometric in nature.  The purpose of this section is to recall their construction and to introduce some related series that satisfy better $p$-adic estimates.  

Let $L\subseteq{\mathbb Z}^{N}$ be the lattice of relations on the set $A$:
\[ L= \bigg\{l=(l_1,\dots,l_N)\in{\mathbb Z}^N\mid \sum_{j=1}^N l_j{\bf a}_j = {\bf 0}\bigg\}. \]
For each $l=(l_1,\dots,l_N)\in L$, we define a partial differential operator $\Box_l$ in variables $\{\Lambda_j\}_{j=1}^N$ by
\begin{equation}
\Box_l = \prod_{l_j>0} \bigg(\frac{\partial}{\partial\Lambda_j}\bigg)^{l_j} - \prod_{l_j<0} \bigg(\frac{\partial}{\partial\Lambda_j}\bigg)^{-l_j}.
\end{equation}
For $\beta = (\beta_0,\beta_1,\dots,\beta_{n+1})\in{\mathbb C}^{n+2}$, the corresponding Euler (or homogeneity) operators are defined by
\begin{equation}
 Z_i = \sum_{j=1}^N a_{ij}\Lambda_j\frac{\partial}{\partial\Lambda_j} - \beta_i 
\end{equation}
for $i=0,\dots,n+1$.  The {\it $A$-hypergeometric system with parameter $\beta$\/} consists of Equations (7.1) for $l\in L$ and (7.2) for $i=0,1,\dots,n+1$.

\begin{comment}
We begin by constructing formal logarithmic solutions of the $A$-hypergeometric system with parameter $\beta = {\bf 0}$.  
Define a sequence $\{f_l(t)\}_{l\in{\mathbb Z}}$ of functions of one variable $t$ by the formulas $f_0(t) = \log t$,
\[ f_l(t) = \frac{t^l}{l!}\bigg( \log t - \big(1+\frac{1}{2} + \cdots + \frac{1}{l}\big)\bigg) \quad\text{for $l>0$,} \]
and
\[ f_l(t) = (-1)^{-l-1}(-l-1)!t^l\quad\text{for $l<0$.} \]
Note that this sequence satisfies 
\begin{equation}
\frac{d}{dt}\big(f_l(t)\big) = f_{l-1}(t)\quad \text{for all $l\in {\mathbb Z}$.}
\end{equation}
We combine these functions into a generating series
\[ f(t): = \sum_{l\in{\mathbb Z}} f_l(t). \]
\end{comment}

Consider the formal series 
\[ q(t) = \sum_{l=0}^\infty (-1)^l l!t^{-l-1}. \]
Let $i\in\{1,\dots,M\}$.  The $A$-hypergeometric series of interest arise as the coefficients of $\gamma_0^{u_{n+1}}x^u$ in the expression
\begin{equation}
{F}^{(i)}(\Lambda,x) := \delta_-\bigg(q(\gamma_0\Lambda_i x^{{\bf a}_i}) \prod_{\substack{k=1\\ k\neq i}}^N \exp(\gamma_0\Lambda_kx^{{\bf a}_k})\bigg).
\end{equation}
For $u\in {\mathcal M}_-$, let $F_u^{(i)}(\Lambda)$ be the coefficient of $\gamma_0^{u_{n+1}}x^u$ in~(7.3) and write
\begin{equation}
{F}^{(i)}(\Lambda,x) = \sum_{u\in {\mathcal M}_-} F^{(i)}_u(\Lambda)\gamma_0^{u_{n+1}}x^u.
\end{equation}
If we set
\[ L_{i,u} = \bigg\{ l=(l_1,\dots,l_N)\in{\mathbb Z}^N\mid \text{$\sum_{k=1}^N l_k{\bf a}_k = u$, $l_i\leq 0$, and $l_k\geq 0$ for $k\neq i$}\bigg\},  \]
then from (7.3)
\begin{equation}
F_u^{(i)}(\Lambda) =  \sum_{l=(l_1,\dots,l_N)\in L_{i,u}} \frac{(-1)^{-l_i-1}(-l_i-1)!}{\displaystyle \prod_{\substack{k=1\\ k\neq i}}^N l_k!} \prod_{k=1}^N \Lambda_k^{l_k}.
\end{equation}

\begin{comment}
For each $i=1,\dots,M$ define a formal generating series $F^{(i)}(\Lambda,x)$ by the formula
\begin{equation}
F^{(i)}(\Lambda,x) = \delta_-\bigg( f(\gamma_0\Lambda_i x^{{\bf a}_i}) \prod_{\substack{k=1\\ k\neq i}}^N \exp(\gamma_0\Lambda_kx^{{\bf a}_k})\bigg)
\end{equation}
Observe that the coefficient of $x^{\bf 0}$ in the product on the right-hand side of (6.4) is 
\begin{equation}
F^{(i)}_{\bf 0}(\Lambda) = \log\gamma_0\Lambda_i + \sum_{\substack{l=(l_1,\dots,l_N)\in L_i\\ l_i<0}} \frac{(-1)^{-l_i-1}(-l_i-1)!}{\displaystyle \prod_{\substack{k=1\\ k\neq i}}^N l_k!}\prod_{k=1}^N \Lambda_k^{l_k}.
\end{equation}

For $u\in {\mathcal M}_-$ put
\[ L_{i,u} = \bigg\{ l=(l_1,\dots,l_N)\in{\mathbb Z}^N\mid \text{$\sum_{k=1}^N l_k{\bf a}_k = u$, $l_i\leq 0$, and $l_k\geq 0$ for $k\neq i$}\bigg\}.  \]
One calculates that
\begin{equation}
 F^{(i)}(\Lambda,x) = \sum_{u\in {\mathcal M}_-} F^{(i)}_u(\Lambda) \gamma_0^{u_{n+1}}x^u,
\end{equation}
where
\begin{equation}
F^{(i)}_u(\Lambda) = \sum_{l=(l_1,\dots,l_N)\in L_{i,u}} \frac{(-1)^{-l_i-1}(-l_i-1)!}{\displaystyle \prod_{\substack{k=1\\ k\neq i}}^N l_k!} \prod_{k=1}^N \Lambda_k^{l_k}.
\end{equation}
\end{comment}

It follows from the definition of $q(t)$ that for $i=1,\dots,M$
\[ \frac{\partial}{\partial\Lambda_i} q(\gamma_0\Lambda_ix^{{\bf a}_i}) = \gamma_0x^{{\bf a}_i}q(\gamma_0\Lambda_ix^{{\bf a}_i})-\frac{1}{\Lambda_i}. \]
A straightforward calculation from (7.3) gives
\begin{equation}
\frac{\partial}{\partial\Lambda_k}{F}^{(i)}(\Lambda,x) = \delta_-\big(\gamma_0x^{{\bf a}_k}{F}^{(i)}(\Lambda,x)\big)
\end{equation}
for $k=1,\dots,N$.  Equivalently, for $u\in {\mathcal M}_-$, we have by (7.6)
\begin{equation}
\frac{\partial}{\partial\Lambda_k}F^{(i)}_u(\Lambda) = F^{(i)}_{u-{\bf a}_k}(\Lambda).
\end{equation}
More generally, if $l_1,\dots,l_N$ are nonnegative integers, it follows that
\begin{equation}
\prod_{k=1}^N \bigg(\frac{\partial}{\partial\Lambda_k}\bigg)^{l_k} F^{(i)}_u(\Lambda) = F_{u-\sum_{k=1}^N l_k{\bf a}_k}^{(i)}(\Lambda).
\end{equation}
In particular, it follows from the definition of the box operators (7.1) that
\begin{equation}
\Box_l\big(F^{(i)}_u(\Lambda)\big) = 0\quad\text{for all $l\in L$, $i=1,\dots,M$, and all $u\in {\mathcal M}_-$.}
\end{equation}
The condition on the summation in (7.5) implies that $F^{(i)}_u(\Lambda)$ satisfies the Euler operators (7.2) with parameter $\beta = u$.  
In summary:
\begin{lemma}
For $i=1,\dots,M$ and $u\in{\mathcal M}_-$, the series $F^{(i)}_u(\Lambda)$ is a solution of the $A$-hypergeometric system with parameter $\beta = u$.
\end{lemma}

\begin{lemma}
For $i=1,\dots,M$ and $u\in {\mathcal M}_-$, the series $F^{(i)}_u(\Lambda)$ lies in~$R_u$.
\end{lemma}

\begin{proof}
The homogeneity condition is clear from (7.5).  We prove that $F^{(i)}_u(\Lambda)$ lies in $\tilde{R}$.  

Let $u\in {\mathcal M}_-$.  By our choice of the set $A$, we can write $-u=\sum_{k=1}^N l_k{\bf a}_k$ with the $l_k$ in 
${\mathbb Z}_{\geq 0}$.  Suppose that $(l'_1,\dots,l'_N)\in L_{i,u}$.  Then 
\[ \sum_{k=1}^N (l_k+l'_k){\bf a}_k = {\bf 0}. \]
Furthermore, we have $l_k+l'_k\geq 0$ for $k\neq i$ and $l_i+l'_i\leq 0$, so 
\[ (l_1+l'_1,\dots,l_N+l'_N)\in L_i. \]
It follows that $L_{i,u}\subseteq -(l_1,\dots,l_N)+L_i$, hence $F_u^{(i)}(\Lambda)\in \Lambda_1^{-l_1}\cdots \Lambda_N^{-l_N} R_E$.  
\end{proof}

{\bf Remark.}  We note the relation between the $F^{(i)}_u(\Lambda)$ and the $F_{ij}(\Lambda)$ defined in the Introduction: for $i,j=1,\dots,M$,
\begin{equation}
F_{ij}(\Lambda) = \Lambda_j F^{(i)}_{-{\bf a}_j}(\Lambda)= {\mathbb M}\big(F(\Lambda,x)\big).
\end{equation}

\begin{lemma}
For $i=1,\dots,M$ and $u\in {\mathcal M}_-$, the coefficients of the series $F^{(i)}_u(\Lambda)$ are integers divisible by 
$(-u_{n+1}-1)!$.
\end{lemma}

\begin{proof}
Since the last coordinate of each ${\bf a}_k$ equals $1$, the condition on the summation in (7.5) implies that
\[ \sum_{k=1}^N l_k = u_{n+1}, \]
i.~e., 
\[(-l_i-1)=  (-u_{n+1}-1) +\sum_{\substack{k=1\\ k\neq i}}^N l_k. \]
Applying this formula to the coefficients of the series in (7.5) implies the lemma.
\end{proof}

Although the series $F^{(i)}_u(\Lambda)$ are more natural to consider, we replace them in what follows by some related series $G^{(i)}_u(\Lambda)$ that are more closely tied to the Dwork-Frobenius operator $\alpha^*$.

For $i=1,\dots,M$, define 
\begin{align}
{G}^{(i)}(\Lambda,x) &= \delta_-\big({F}^{(i)}(\Lambda,x)\hat{\theta}_1(\Lambda,x)\big) \\
& =\delta_-\bigg(q(\gamma_0\Lambda_ix^{{\bf a}_i})\hat{\theta}_1(\Lambda_ix^{{\bf a}_i})\bigg(\prod_{\substack{k=1\\k\neq i}}^N \hat{\theta}(\Lambda_kx^{{\bf a}_k})\bigg)\bigg), \nonumber
\end{align}
where the second equality follows from (4.1).  If we write
\begin{equation}
{G}^{(i)}(\Lambda,x) = \sum_{u\in {\mathcal M}_-} G^{(i)}_u(\Lambda) \gamma_0^{u_{n+1}}x^u,
\end{equation}
then by (3.8) and (7.4) one gets from the first equality of (7.14)
\begin{equation}
G^{(i)}_u(\Lambda) = \sum_{\substack{u^{(1)}\in {\mathcal M}_-,u^{(2)}\in{\mathbb N}A\\ u^{(1)}+u^{(2)} = u}} F_{u^{(1)}}^{(i)}(\Lambda)\hat{\theta}_{1,u^{(2)}}(\Lambda).
\end{equation}
Note that as a series in $\Lambda$, $G^{(i)}_u(\Lambda)$ has exponents in $L_{i,u}$.  

We have an analogue of Lemma 7.13 for these series.
\begin{lemma}
For $i=1,\dots,M$ and $u\in {\mathcal M}_-$, $G^{(i)}_u(\Lambda)/(-u_{n+1}-1)!$ has $p$-integral coefficients.
\end{lemma}

\begin{proof}
Using (3.9) we can write (7.16) in the form
\begin{multline}
G^{(i)}_u(\Lambda) = \\ 
\sum_{\substack{u^{(1)}\in {\mathcal M}_-,u^{(2)}\in{\mathbb N}A\\ u^{(1)}+u^{(2)} = u}} \frac{F_{u^{(1)}}^{(i)}(\Lambda)}{(-u^{(1)}_{n+1}-1)!} \sum_{\substack{l_1,\dots,l_N\in{\mathbb N}\\ \sum_{k=1}^N l_k{\bf a}_k = u^{(2)}}} \bigg(\prod_{k=1}^N
\hat{\theta}_{1,l_k}\bigg) \frac{(-u^{(1)}_{n+1}-1)!}{\prod_{k=1}^N l_k!} \Lambda_1^{l_1}\cdots\Lambda_N^{l_N}.
\end{multline}
The series $F_{u^{(1)}}^{(i)}(\Lambda)/(-u^{(1)}_{n+1}-1)!$ has integral coefficients by Lemma~7.13.  The condition on the first summation on the right-hand side of (7.18) implies that
\begin{equation}
(-u_{n+1}-1)+ u^{(2)}_{n+1} = -u^{(1)}_{n+1} -1, 
\end{equation}
where $-u_{n+1},-u^{(1)}_{n+1}\geq 1$ since $u,u^{(1)}\in {\mathcal M}_-$.
Since the last coordinate of each ${\bf a}_k$ equals 1, the condition on the second summation on the right-hand side of (7.18) implies that
\[ u^{(2)}_{n+1} = \sum_{k=1}^N l_k, \]
so by (7.19)
\[ (-u_{n+1}-1) + \sum_{k=1}^N l_k= -u^{(1)}_{n+1}-1. \]
It follows that the ratio $(-u^{(1)}_{n+1}-1)!/\prod_{k=1}^N l_k$ appearing in the second summation on the right-hand side of (7.18) is an integer divisible by $(-u_{n+1}-1)!$.  For each $N$-tuple $(l_1,\dots,l_N)$ appearing in the second summation on the right-hand side of (7.18) we have
\begin{equation}
{\rm ord}\:\prod_{k=1}^N \hat{\theta}_{1,l_k}\geq \frac{\sum_{k=1}^N l_k(p-1)}{p} = \frac{u^{(2)}_{n+1}(p-1)}{p}
\end{equation}
by (3.6).  This implies that the series on the right-hand side of (7.18) converges to a series in the $\Lambda_k$ with $p$-integral coefficients that remain $p$-integral when divided by $(-u_{n+1}-1)!$.
\end{proof}

We also have an analogue of Lemma 7.11.
\begin{lemma}
For $i=1,\dots,M$ and $u\in{\mathcal M}_-$ the series $G^{(i)}_u(\Lambda)$ lies in ${R}_u$.
\end{lemma}

\begin{proof}
The homogeneity condition is clear.  Since the second summation on the right-hand side of (7.18) is finite, the sum
\begin{equation}
\frac{F_{u^{(1)}}^{(i)}(\Lambda)}{(-u^{(1)}_{n+1}-1)!} \sum_{\substack{l_1,\dots,l_N\in{\mathbb N}\\ \sum_{k=1}^N l_k{\bf a}_k = u^{(2)}}} \bigg(\prod_{k=1}^N
\hat{\theta}_{1,l_k}\bigg) \frac{(-u^{(1)}_{n+1}-1)!}{\prod_{k=1}^N l_k!} \Lambda_1^{l_1}\cdots\Lambda_N^{l_N}
\end{equation}
lies in $\tilde{R}$ by the proof of Lemma 7.11.  Furthermore, it follows from (7.20) that the expression (7.22) has norm
bounded by $p^{-u^{(2)}_{n+1}(p-1)/p}$.  In the first summation on the right-hand side of (7.18), a given $u^{(2)}\in{\mathbb N}A$ can appear at most once, i.~e, $u^{(2)}\to\infty$ in this summation. This implies that the first summation on the right-hand side of (7.18) converges in the norm on ${\mathbb B}$.
\end{proof}

We define a matrix $G(\Lambda) = \big[ G_{ij}(\Lambda) \big]_{i,j=1}^M$ corresponding to $F(\Lambda)$ by the analogue of~(7.12):
\begin{equation}
G_{ij}(\Lambda) = \Lambda_j G^{(i)}_{-{\bf a}_j}(\Lambda) = {\mathbb M}\big(G(\Lambda,x)\big).
\end{equation}
Like the $F_{ij}(\Lambda)$, the monomials in the $G_{ij}(\Lambda)$ have exponents in $L_i$.  Furthermore, the $G_{ii}(\Lambda)$ have constant term~1, while the $G_{ij}(\Lambda)$ for $i\neq j$ have no constant term.  This implies that $\det G(\Lambda)$ is an invertible element of~${R}$, hence $G(\Lambda)^{-1}$ is a matrix with entries in ${R}$.  Since the series $G_{ij}(\Lambda)$ have integral coefficients (Lemma~7.17) and the $G_{ii}(\Lambda)$ have constant term~1, it follows that
\begin{equation} 
\big\lvert G(\Lambda) \big\rvert = \big\lvert \det G(\Lambda)\big\rvert = 1. 
\end{equation}
This implies in particular that 
\begin{equation}
\lvert G(\Lambda)^{-1}\rvert\leq 1.
\end{equation}  

We simplify our notation: for $u,u^{(1)}\in{\mathcal M}_-$, $u\neq u^{(1)}$, set
\[ C_{u,u^{(1)}} = \sum_{\substack{l_1,\dots,l_N\in{\mathbb N}\\ \sum_{k=1}^N l_k{\bf a}_k = u-u^{(1)}}} \bigg(\prod_{k=1}^N
\hat{\theta}_{1,l_k}\bigg) \frac{(-u^{(1)}_{n+1}-1)!}{\prod_{k=1}^N l_k!} \Lambda_1^{l_1}\cdots\Lambda_N^{l_N}. \]
Note that this is a finite sum, $C_{u,u^{(1)}}$ is $p$-integral, and ${\rm ord}\:C_{u,u^{(1)}}>0$ by~(7.20).  Equation (7.18) then simplifies to
\begin{equation}
G^{(i)}_u(\Lambda) = F^{(i)}_u(\Lambda) + 
\sum_{\substack{u^{(1)}\in {\mathcal M}_-\\ u^{(1)}\neq u}} C_{u,u^{(1)}}\frac{F_{u^{(1)}}^{(i)}(\Lambda)}{(-u^{(1)}_{n+1}-1)!}
\end{equation}
Furthermore, the estimate (7.20) implies that $C_{u,u^{(1)}}\to 0$ as $u^{(1)}\to\infty$ in the sense that for any $\kappa>0$, the estimate ${\rm ord}\: C_{u,u^{(1)}}>\kappa$ holds for all but finitely many~$u^{(1)}$.  

\begin{comment}
We define (cf.\ (7.4))
\begin{equation}
{F}^{(i)}(\Lambda,x) = \delta_-\bigg(f(\gamma_0\Lambda_i x^{{\bf a}_i}) \prod_{\substack{k=1\\ k\neq i}}^N \exp(\gamma_0\Lambda_kx^{{\bf a}_k})\bigg) = \sum_{u\in {\mathcal M}_-} F^{(i)}_u(\Lambda)\gamma_0^{u_{n+1}}x^u,
\end{equation}
where the $F^{(i)}_u(\Lambda)$ are given by (7.6).  We also define (cf.\ (7.17)) 
\begin{equation}
G^{(i)}(\Lambda,x) = \delta_-\big({F}^{(i)}(\Lambda,x)\hat{\theta}_1(\Lambda,x)\big) = \sum_{u\in {\mathcal M}_-} G^{(i)}_u(\Lambda) \gamma_0^{u_{n+1}}x^u,
\end{equation}
where the $G^{(i)}_u(\Lambda)$ are given by (7.19).  Note that when $u\in{\mathcal M}_-$, one must have $u^{(1)}\in{\mathcal M}_-$ on the right-hand side of (7.19) to satisfy the equation $u^{(1)} + u^{(2)} = u$:
\begin{equation}
G^{(i)}_u(\Lambda) = \sum_{\substack{u^{(1)}\in {\mathcal M}_-,u^{(2)}\in{\mathbb N}A\\ u^{(1)}+u^{(2)} = u}} F_{u^{(1)}}^{(i)}(\Lambda)\hat{\theta}_{1,u^{(2)}}(\Lambda).
\end{equation}
\end{comment}

We need the analogue of (7.26) with the roles of $F$ and $G$ reversed.  It follows from (7.14) and (4.1) that for $i=1,\dots,M$
\begin{equation}
F^{(i)}(\Lambda,x) = \delta_-\big(G^{(i)}(\Lambda,x) \hat{\theta}_1(\Lambda,x)^{-1}\big).
\end{equation}
This leads to the analogue of (7.18):
\begin{multline}
F^{(i)}_u(\Lambda) = \\ 
\sum_{\substack{u^{(1)}\in {\mathcal M}_-,u^{(2)}\in{\mathbb N}A\\ u^{(1)}+u^{(2)} = u}} \frac{G_{u^{(1)}}^{(i)}(\Lambda)}{(-u^{(1)}_{n+1}-1)!} \sum_{\substack{l_1,\dots,l_N\in{\mathbb N}\\ \sum_{k=1}^N l_k{\bf a}_k = u^{(2)}}} \bigg(\prod_{k=1}^N
\hat{\theta}'_{1,l_k}\bigg) \frac{(-u^{(1)}_{n+1}-1)!}{\prod_{k=1}^N l_k!} \Lambda_1^{l_1}\cdots\Lambda_N^{l_N},
\end{multline}
where the $\hat{\theta}'_{1,l_k}$ are defined by (3.10) and Lemma 7.17 tells us that the 
\[ {G_{u^{(1)}}^{(i)}(\Lambda)}/{(-u^{(1)}_{n+1}-1)!}\]  
have $p$-integral coefficients.  We define for $u,u^{(1)}\in{\mathcal M}_-$, $u\neq u^{(1)}$, 
\[ C'_{u,u^{(1)}} = \sum_{\substack{l_1,\dots,l_N\in{\mathbb N}\\ \sum_{k=1}^N l_k{\bf a}_k = u-u^{(1)}}} \bigg(\prod_{k=1}^N
\hat{\theta}'_{1,l_k}\bigg) \frac{(-u^{(1)}_{n+1}-1)!}{\prod_{k=1}^N l_k!} \Lambda_1^{l_1}\cdots\Lambda_N^{l_N}. \]
Since the $\hat{\theta}'_{1,l_k}$ also satisfy the estimate (7.20) (see (3.11)), we get that $C'_{u,u^{(1)}}$ is $p$-integral and 
${\rm ord}\:C'_{u,u^{(1)}}>0$.  In addition, $C'_{u,u^{(1)}}\to 0$ as $u^{(1)}\to\infty$.  Substitution into (7.28) now gives the desired formula:
\begin{equation}
F^{(i)}_u(\Lambda) = G^{(i)}_u(\Lambda) + 
\sum_{\substack{u^{(1)}\in {\mathcal M}_-\\ u^{(1)}\neq u}} C'_{u,u^{(1)}}\frac{G_{u^{(1)}}^{(i)}(\Lambda)}{(-u^{(1)}_{n+1}-1)!}
\end{equation}

For future reference, we record estimates that were used in the proof of (7.26) and (7.29).
\begin{proposition}
For all $u,u^{(1)}\in {\mathcal M}_-$, $u\neq u^{(1)}$, we have ${\rm ord}\:C_{u,u^{(1)}}>0$ (resp.\ ${\rm ord}\:C'_{u,u^{(1)}}>0$) and $C_{u,u^{(1)}}\to 0$ (resp. $C'_{u,u^{(1)}}\to 0$) as $u^{(1)}\to\infty$.
\end{proposition}

\section{Eigenvectors of $\alpha^*$}

We use the series of Section~7 to construct eigenvectors of $\alpha^*$.  These eigenvectors will lead to the fixed point of the contraction mapping of Section~6.
We begin by recalling some results from~\cite{AS1}.

By \cite[Lemma 6.1]{AS1} the product $\hat{\theta}_1(t)q(\gamma_0t)$ is well-defined so we may set
\begin{equation}
Q(t) = \delta_-\big(\hat{\theta}_1(t)q(\gamma_0t)\big) = \sum_{i=1}^\infty Q_i i!\gamma_0^{-i-1} t^{-i-1},
\end{equation}
where the $Q_i$ are defined by the second equality.  The proof of \cite[Lemma 6.1]{AS1} shows that the $Q_i$ are $p$-integral.  By \cite[Proposition~6.10]{AS1} we have
\begin{equation}
\delta_-\big(\theta(t)Q(t^p)\big) = pQ(t).
\end{equation}

It follows from (8.2) that 
\[ \theta(t)Q(t^p) = pQ(t) + A(t) \]
for some series $A(t)$ in nonnegative powers of $t$.  Fix $i$, $1\leq i\leq M$ and replace $t$ in this equation by $\Lambda_ix^{{\bf a}_i}$:
\[ \theta(\Lambda_ix^{{\bf a}_i}) Q(\Lambda_i^p x^{p{\bf a}_i}) = A(\Lambda_ix^{{\bf a}_i}) + pQ(\Lambda_i x^{{\bf a}_i}). \]
Since $\delta_-(Q(\Lambda_i x^{{\bf a}_i})) = Q(\Lambda_i x^{{\bf a}_i})$ and $\delta_-(A(\Lambda_ix^{{\bf a}_i})) = 0$, we get
\begin{equation}
\delta_-\bigg(  \theta(\Lambda_ix^{{\bf a}_i}) Q(\Lambda_i^p x^{p{\bf a}_i}) \bigg) = pQ(\Lambda_i x^{{\bf a}_i}).
\end{equation}

These series are related to the $G^{(i)}(\Lambda,x)$ defined in Section~7.
\begin{lemma}
We have
\[ G^{(i)}(\Lambda,x) = \delta_-\bigg( Q(\Lambda_i x^{{\bf a}_i}) \prod_{\substack{j=1\\ j\neq i}}^N \hat{\theta}(\Lambda_jx^{{\bf a}_j})\bigg). \]
\end{lemma}

\begin{proof}
Combining (7.14), (7.3), (3.7) and using (4.1) we get
\[ G^{(i)}(\Lambda,x) = \delta_-\bigg( q(\gamma_0\Lambda_ix^{{\bf a}_i})\prod_{\substack{j=1\\ j\neq i}}^N \exp(\gamma_0\Lambda_jx^{{\bf a}_j})\prod_{j=1}^N \hat{\theta}_1(\Lambda_jx^{{\bf a}_j})\bigg). \]
Using (3.3) and (3.5), this may be rewritten as
\[ G^{(i)}(\Lambda,x) = \delta_-\bigg( q(\gamma_0\Lambda_ix^{{\bf a}_i}) \hat{\theta}_1(\Lambda_ix^{{\bf a}_i}) \prod_{\substack{j=1\\ j\neq i}}^N \hat{\theta}(\Lambda_jx^{{\bf a}_j})\bigg). \]
The assertion of the lemma now follows from (8.1) and (4.1).
\end{proof}

The following result is a key step in the proof of Theorem~2.11.
\begin{theorem}
For $i=1,\dots,M$ we have
\[ \alpha^*\big(G^{(i)}(\Lambda,x)\big) = pG^{(i)}(\Lambda,x). \]
\end{theorem}

\begin{proof}
Since $\theta(t) = \hat{\theta}(t)/\hat{\theta}(t^p)$ we have
\begin{equation}
\prod_{\substack{j=1\\ j\neq i}}^N \theta(\Lambda_jx^{{\bf a}_j})\prod_{\substack{j=1\\ j\neq i}}^N \hat{\theta}(\Lambda_j^px^{p{\bf a}_j}) = \prod_{\substack{j=1\\ j\neq i}}^N \hat{\theta}(\Lambda_jx^{{\bf a}j}).
\end{equation}
We now compute
\begin{align*}
\alpha^*\big(G^{(i)}(\Lambda,x)\big) & = \delta_-\bigg(\prod_{j=1}^N \theta(\Lambda_jx^{{\bf a}_j}) \delta_-\bigg( Q(\Lambda_i^p x^{p{\bf a}_i}) \prod_{\substack{j=1\\ j\neq i}}^N \hat{\theta}(\Lambda_j^px^{p{\bf a}_j})\bigg)\bigg) \\
 & = \delta_-\bigg(\theta(\Lambda_ix^{{\bf a}_i})Q(\Lambda_i^px^{p{\bf a}_i}) \prod_{\substack{j=1\\ j\neq i}}^N \theta(\Lambda_jx^{{\bf a}_j}) \prod_{\substack{j=1\\ j\neq i}}^N \hat{\theta}(\Lambda_j^px^{p{\bf a}_j})\bigg) \\
 & = p\delta_-\bigg(Q(\Lambda_ix^{{\bf a}_i}) \prod_{\substack{j=1\\ j\neq i}}^N \hat{\theta}(\Lambda_jx^{{\bf a}_j})\bigg) \\
& = pG^{(i)}(\Lambda,x),
\end{align*}
where the first equality follows from Lemma 8.4, the second follows from (4.1) and rearranging the product, the third follows from (8.3) and (8.6), and the fourth follows from Lemma 8.4.  Note that the rearrangement of the product occuring in the second equality is not completely trivial.  One of the series contains negative powers of the $x_i$, the rest contain positive powers of the $x_i$, so they do not all lie in a common commutative ring.  One has to write out both sides to verify that they are equal.  See \cite[Section 5]{AS2} for more details on this calculation.  
\end{proof}

The action of $\alpha^*$ was extended to $S^M$ coordinatedwise, so if we set
\begin{equation}
G(\Lambda,x) = \big( G^{(1)}(\Lambda,x),\dots,G^{(M)}(\Lambda,x)\big)\in S^M,
\end{equation}
then we get the following result.
\begin{corollary}
We have $\alpha^*\big(G(\Lambda,x)\big) = pG(\Lambda,x)$.
\end{corollary}

We now verify that $G(\Lambda)^{-1}G(\Lambda,x)\in T$.  First of all, 
\[ {\mathbb M}\big(G(\Lambda)^{-1}G(\Lambda,x)\big) = I \]
by (7.23) and (5.5).  This implies in particular that 
\[ \lvert G(\Lambda)^{-1}G(\Lambda,x)\rvert\geq 1. \]
But $\lvert G(\Lambda,x)\rvert\leq 1$ by Lemma 7.17 and $\lvert G(\Lambda)^{-1}\rvert\leq 1$ by~(7.25), so
\[ \lvert G(\Lambda)^{-1}G(\Lambda,x)\rvert\leq 1 \]
also.  We thus conclude that
\begin{equation}
\lvert G(\Lambda)^{-1}G(\Lambda,x)\rvert= 1. 
\end{equation}

\begin{corollary}
The product $G(\Lambda)^{-1}G(\Lambda,x)$ is the unique fixed point of $\phi$, hence lies in $T'$.
\end{corollary}

\begin{proof}
From the definition of $\alpha^*$ and Corollary 8.8 we have
\begin{align}
\alpha^*\big(G(\Lambda)^{-1}G(\Lambda,x)\big)  &= G(\Lambda^p)^{-1} \alpha^*\big(G(\Lambda,x)\big) \\
 &= G(\Lambda^p)^{-1}pG(\Lambda,x). \nonumber
\end{align}
This implies by (7.23) and (5.4) that 
\begin{equation}
 {\mathbb M}\big(\alpha^*\big(G(\Lambda)^{-1}G(\Lambda,x)\big)\big) = pG(\Lambda^p)^{-1}G(\Lambda), 
\end{equation}
hence
\[ \phi\big(G(\Lambda)^{-1}G(\Lambda,x)\big) = G(\Lambda)^{-1}G(\Lambda,x). \]
\end{proof}

\begin{corollary}
The entries of $G(\Lambda^p)^{-1}G(\Lambda)$ lie in $R'_{\bf 0}$.
\end{corollary}

\begin{proof} 
Corollary 8.10 implies that $G(\Lambda)^{-1}G(\Lambda,x)$ lies in $(S')^M$, and since $\alpha^*$ is stable on $(S')^M$ we also have $\alpha^*\big(G(\Lambda)^{-1}G(\Lambda,x)\big)\in (S')^M$.  The assertion of the corollary now follows from (8.11).
\end{proof}

Put ${\mathcal G}(\Lambda) = G(\Lambda^p)^{-1}G(\Lambda)$.  We note one more consequence of Theorem~8.5.
\begin{proposition}
The reciprocal eigenvalues of ${\mathcal G}(\lambda)$ are $p$-adic units for all $\lambda$ in~${\mathcal D}$.
\end{proposition}

\begin{proof}
We showed earlier that $G(\Lambda)^{-1}G(\Lambda,x)\in T$, so $G(\Lambda)^{-1}G(\Lambda,x)$ satisfies the hypotheses of Lemma~5.6.  By (8.9) and (5.8) we have
\[ \big\lvert \det{\mathbb M}\big(\alpha^*\big(G(\Lambda)^{-1}G(\Lambda,x)\big)\big) \big\rvert = \lvert p\rvert \]
and by (8.9) and Corollary 5.10 we have
\[ \big\lvert \det{\mathbb M}\big(\alpha^*\big(G(\Lambda)^{-1}G(\Lambda,x)\big)\big)^{-1} \big\rvert = \lvert 1/p\rvert. \]
Combining these equalities with (8.12) gives
\[ \big\lvert\det\big(G(\Lambda^p)^{-1}G(\Lambda)\big)\big\rvert =\big\lvert\det\big(G(\Lambda)^{-1}G(\Lambda^p)\big)\big\rvert = 1. \]
By (2.12), we then have
\[ \lvert\det{\mathcal G}(\lambda)\rvert\leq 1 \;\text{and}\; \lvert\det{\mathcal G}(\lambda)^{-1}\rvert\leq 1\;\text{for all $\lambda\in{\mathcal D}$}, \]
which implies that $\det{\mathcal G}(\lambda)$ assumes unit values on ${\mathcal D}$.  Since ${\mathcal G}(\lambda)$ assumes $p$-integral values on ${\mathcal D}$, it follows that the roots of $\det(I-t{\mathcal G}(\lambda))$ are $p$-adic units.
\end{proof}

We introduce some additional notation that will be useful later.  Put
\begin{equation}
 {\mathcal G}(\Lambda,x) = \big({\mathcal G}^{(1)}(\Lambda,x),\dots,{\mathcal G}^{(M)}(\Lambda,x)\big) = G(\Lambda)^{-1}G(\Lambda,x). 
\end{equation}
This is an element of $T'$ by Corollary 8.10, so each ${\mathcal G}^{(i)}(\Lambda,x)$ lies in $S'$ and we may write
\begin{equation}
{\mathcal G}^{(i)}(\Lambda,x) = \sum_{u\in{\mathcal M}_-} {\mathcal G}_u^{(i)}(\Lambda) \gamma_0^{u_{n+1}} x^u 
\end{equation}
with ${\mathcal G}_u^{(i)}(\Lambda)$ in $R_u'$ for all $i,u$.  

\section{Relation between ${\mathcal F}(\Lambda)$ and ${\mathcal G}(\Lambda)$}

In this section we prove the first assertion of Theorem~2.11 and show that the second assertion of Theorem~2.11 is equivalent to the same statement with ${\mathcal F}(\Lambda)$ replaced by ${\mathcal G}(\Lambda)$ (see Proposition ~9.5).

Put $F(\Lambda,x) = (F^{(1)}(\Lambda,x),\dots,F^{(M)}(\Lambda,x))$, where $F^{(i)}(\Lambda,x)$ is given by~(7.3).  We regard $F(\Lambda,x)$ and $G(\Lambda,x)$ (defined in (8.7)) as column vectors so they can be multiplied on the left by $(M\times M)$-matrices.  The following result is a consequence of the fact that $G(\Lambda)^{-1}G(\Lambda,x)$ lies in $(S')^M$ (Corollary~8.10).  
\begin{proposition}
The vectors $F(\Lambda)^{-1}F(\Lambda,x)$,  $G(\Lambda)^{-1}F(\Lambda,x)$, and $F(\Lambda)^{-1}G(\Lambda,x)$ lie in $(S')^M$.
\end{proposition}

\begin{proof}
We begin by showing that $G(\Lambda)^{-1}F(\Lambda,x)\in (S')^M$.  
Writing Eq.~(7.29) for $i=1,\dots,M$ and multiplying by $G(\Lambda)^{-1}$ gives the vector equation
\begin{multline}
G(\Lambda)^{-1}\begin{bmatrix} F^{(1)}_u(\Lambda) \\ \vdots \\ F^{(M)}_u(\Lambda)\end{bmatrix} = \\
G(\Lambda)^{-1} \begin{bmatrix} G^{(1)}_u(\Lambda) \\ \vdots \\ G^{(M)}_u(\Lambda)\end{bmatrix}+ 
\sum_{\substack{u^{(1)}\in {\mathcal M}_-\\ u^{(1)}\neq u}} \frac{C'_{u,u^{(1)}}}{(-u^{(1)}_{n+1}-1)!}G(\Lambda)^{-1} 
\begin{bmatrix} G^{(1)}_{u^{(1)}}(\Lambda) \\ \vdots \\ G^{(M)}_{u^{(1)}}(\Lambda)\end{bmatrix}.
\end{multline}
We have $\lvert G(\Lambda)^{-1}\rvert\leq 1$ by (7.25).  Lemma~7.17, Proposition~7.30, and Corollary~8.10 imply that the sum on the right-hand side of (9.2) converges to an element of $(R'_u)^M$.  
Since $u\in{\mathcal M}_-$ was arbitrary, this shows that $G(\Lambda)^{-1}F(\Lambda,x)\in (S')^M$.  

Take successively $u=-{\bf a}_j$, $j=1,\dots,M$ in (9.2) and multiply the $j$-th equation by $\Lambda_j$.  We combine the column vectors in the resulting $M$ equations into matrices and apply (7.12) and (7.23).  The resulting matrix equation is 
\begin{equation}
G(\Lambda)^{-1}F(\Lambda) = I+ H(\Lambda),
\end{equation}
where we have written $H(\Lambda)$ for the $(M\times M)$-matrix whose $j$-th column is
\[ \sum_{\substack{u^{(1)}\in{\mathcal M}_-\\ u^{(1)}\neq -{\bf a}_j}} \frac{C'_{-{\bf a}_j,u^{(1)}}\Lambda_j}  {(-u^{(1)}_{n+1}-1)!}
G(\Lambda)^{-1}\begin{bmatrix} G^{(1)}_{u^{(1)}}(\Lambda) \\ \vdots \\ G^{(M)}_{u^{(1)}}(\Lambda)\end{bmatrix}. \]
The entries in the matrix $H(\Lambda)$ all lie in ${R}'_{\bf 0}$ by Corollary~8.10 and have norm $<1$ by (7.25), Lemma~7.17, and Proposition~7.30.  We can thus apply the usual geometric series formula to invert the right-hand side of (9.3).  This proves that $F(\Lambda)^{-1}G(\Lambda)$ is a matrix with entries in ${R}'_{\bf 0}$, all entries having norm $\leq 1$.  Since the left-hand side of (9.2) lies in $({R}'_u)^M$, we can now multiply it by $F(\Lambda)^{-1}G(\Lambda)$ to conclude that 
\begin{equation}
F(\Lambda)^{-1}\begin{bmatrix} F^{(1)}_u(\Lambda) \\ \vdots \\ F^{(M)}_u(\Lambda)\end{bmatrix}\in ({R}'_u)^M\quad\text{for all $u\in{\mathcal M}_-$}.
\end{equation}
This shows that $F(\Lambda)^{-1}F(\Lambda,x)$ lies in $(S')^M$.

The remaining assertion, that $F(\Lambda)^{-1}G(\Lambda,x)$ lies in $(S')^M$, can be proved similarly by reversing the roles of $F$ and $G$ and using (7.26) in place of (7.29).  Since that result is not needed in what follows, we omit the details.
\end{proof}

\begin{proposition}
The matrix ${\mathcal F}(\Lambda)$ has entries in $R'_{\bf 0}$.  For any $\lambda\in({\mathbb F}_q^\times)^M\times{\mathbb F}_q^{N-M}$ with $\bar{D}(\lambda)\neq 0$ we have
\[ \det\big(I- t{\mathcal F}(\hat{\lambda}^{p^{a-1}}){\mathcal F}(\hat{\lambda}^{p^{a-2}})\cdots{\mathcal F}(\hat{\lambda})\big) = 
\det\big(I-t{\mathcal G}(\hat{\lambda}^{p^{a-1}}){\mathcal G}(\hat{\lambda}^{p^{a-2}})\cdots{\mathcal G}(\hat{\lambda})\big) \]
where $\hat{\lambda}\in {\mathbb Q}_p(\zeta_{q-1})^N$ denotes the Teichm\"uller lifting of $\lambda$.
\end{proposition}

\begin{proof}
We showed in the proof of Proposition 9.1 that the matrix ${\mathcal H}(\Lambda):=F(\Lambda)^{-1}G(\Lambda)$ and its inverse ${\mathcal H}(\Lambda)^{-1} = G(\Lambda)^{-1}F(\Lambda)$ have entries in~${R}_{\bf 0}'$.  We then have
\begin{equation}
{\mathcal F}(\Lambda) = F(\Lambda^p)^{-1}F(\Lambda) = {\mathcal H}(\Lambda^p)\big(G(\Lambda^p)^{-1}G(\Lambda)\big) {\mathcal H}(\Lambda)^{-1} = {\mathcal H}(\Lambda^p){\mathcal G}(\Lambda){\mathcal H}(\Lambda)^{-1}, 
\end{equation}
which implies the first assertion of the proposition since ${\mathcal G}(\Lambda)$ has entries in $R_{\bf 0}'$ by Corollary~8.13.  Eq.~(9.6) implies
\begin{equation}
{\mathcal H}(\Lambda^{p^a})^{-1}{\mathcal F}(\Lambda^{p^{a-1}}){\mathcal F}(\Lambda^{p^{a-2}})\cdots{\mathcal F}(\Lambda){\mathcal H}(\Lambda) = {\mathcal G}(\Lambda^{p^{a-1}}){\mathcal G}(\Lambda^{p^{a-2}})\cdots{\mathcal G}(\Lambda). 
\end{equation}
The Teichm\"uller lifting $\hat{\lambda}$ satisfies $\hat{\lambda}^{p^a} = \hat{\lambda}$ and ${\mathcal H}(\Lambda)$ is a function on ${\mathcal D}$, so the second assertion of the proposition follows by evaluating (9.7) at $\Lambda = \hat{\lambda}$.  
\end{proof}

The first assertion of Proposition 9.5 implies the first assertion of Theorem~2.11.  The second assertion of Proposition 9.5 shows that the second assertion of Theorem~2.11 is equivalent to the following statement.
\begin{theorem}
Let $\lambda\in ({\mathbb F}_q^\times)^M\times{\mathbb F}_q^{N-M}$ and let $\hat{\lambda}\in{\mathbb Q}_p(\zeta_{q-1})^N$ be its Teichm\"uller lifting.  If $\bar{D}(\lambda)\neq 0$, then $\hat{\lambda}^{p^i}\in{\mathcal D}$ for $i=0,\dots,a-1$ and 
\[ \rho(\lambda,t) = \det\big( I-t{\mathcal G}(\hat{\lambda}^{p^{a-1}}){\mathcal G}(\hat{\lambda}^{p^{a-1}})\cdots{\mathcal G}(\hat{\lambda})\big). \]
\end{theorem}

\section{Proof of Theorem 9.8}

We apply Dwork's $p$-adic cohomology theory to prove Theorem 9.8.  

We begin by recalling the formula for the rational function $P_\lambda(t)$ that was proved in \cite{AS1}.  For a subset $I\subseteq\{0,1,\dots,n\}$, let
\begin{multline*}
L^I = \bigg\{ \sum_{u\in{\mathbb N}^{n+2}} c_u\gamma_0^{pu_{n+1}}x^u\mid \text{$\sum_{i=0}^n u_i = du_{n+1}$, $u_i>0$ for $i\in I$, $c_u\in{\mathbb C}_p$,} \\ \text{and $\{c_u\}$ is bounded}\bigg\}. 
\end{multline*}
Put $g_\lambda = x_{n+1}f_\lambda$, so that
\[ g_\lambda(x_0,\dots,x_{n+1}) = \sum_{k=1}^N \lambda_kx^{{\bf a}_k} \in{\mathbb F}_q[x_0,\dots,x_{n+1}], \]
and let $\hat{g}_\lambda$ be its Teichm\"uller lifting:
\[ \hat{g}_\lambda(x_0,\dots,x_{n+1}) = \sum_{k=1}^N\hat{\lambda}_kx^{{\bf a}_k}\in{\mathbb Q}(\zeta_{q-1})[x_0,\dots,x_{n+1}]. \]
From (3.15) we have 
\begin{equation}
\theta(\hat{\lambda},x) = \prod_{j=1}^N \theta(\hat{\lambda}_jx^{{\bf a}_j}). 
\end{equation}
We also need the series $\theta_0(\hat{\lambda},x)$ defined by
\begin{equation}
\theta_0(\hat{\lambda},x) = \prod_{i=0}^{a-1} \prod_{j=1}^N \theta\big((\hat{\lambda}_jx^{{\bf a}_j})^{p^i}\big) = \prod_{i=0}^{a-1}\theta(\hat{\lambda}^{p^i},x^{p^i}).
\end{equation}
Define an operator $\psi$ on formal power series by
\begin{equation}
\psi\bigg(\sum_{u\in{\mathbb N}^{n+2}} c_ux^u\bigg) = \sum_{u\in{\mathbb N}^{n+2}} c_{pu}x^u.
\end{equation}
Denote by $\alpha_{\hat{\lambda}}$ the composition
\[ \alpha_{{\hat{\lambda}}} := \psi^a\circ\text{``multiplication by $\theta_0(\hat{\lambda},x)$.''} \]

The map $\alpha_{\hat{\lambda}}$ is stable on $L^I$ for all $I$ and we have (\cite[Equation~(7.12)]{AS1})
\begin{equation}
P_\lambda(qt) = \prod_{I\subseteq\{0,1,\dots,n\}} 
\det(I-q^{n+1-|I|}t\alpha_{\hat{\lambda}}\mid L_0^I)^{(-1)^{n+1+|I|}}.
\end{equation}
For notational convenience, put $\Gamma = \{0,1,\dots,n\}$.  The following lemma is an immediate consequence of \cite[Proposition~7.13(b)]{AS1} (note that since we are assuming $d\geq n+1$, we have $\mu=0$ in that proposition).
\begin{lemma}
The unit reciprocal roots of $P_\lambda(t)$ are obtained from the reciprocal roots of $\det(I-t\alpha_{\hat{\lambda}}\mid L_0^\Gamma)$ of $q$-ordinal equal to $1$ by division by $q$.
\end{lemma}

We give an alternate description of $\det(I-t\alpha_{\hat{\lambda}}\mid L_0^\Gamma)$ (\cite[Section~7]{AS1}).
Set
\[ B = \bigg\{ \xi^* = \sum_{u\in ({\mathbb Z}_{<0})^{n+2}} c_u^*\gamma_0^{pu_{n+1}}x^u\mid \text{$c_u^*\to 0$ as $u\to-\infty$}\bigg\}, \]
a $p$-adic Banach space with norm $\lvert\xi^*\rvert = \sup_{u\in ({\mathbb Z}_{<0})^{n+2}}\{\lvert c_u^*\rvert\}$. 
Define a map $\Phi$ on formal power series by 
\[ \Phi\bigg(\sum_{u\in{\mathbb Z}^n} c_ux^u\bigg) = \sum_{u\in{\mathbb Z}^n} c_ux^{pu}. \]
Consider the formal composition $\alpha^*_{\hat{\lambda}} = \delta_-\circ\theta_0(\hat{\lambda},x)\circ\Phi^a$.  The following result is \cite[Proposition 7.30]{AS1}.
\begin{proposition}
The operator $\alpha^*_{\hat{\lambda}}$ is an endomorphism of $B$ which is adjoint to 
$\alpha_{\hat{\lambda}}:L_0^\Gamma\to L_0^\Gamma$.
\end{proposition}

From Proposition 10.6, it follows by Serre\cite[Proposition~15]{S} that
\begin{equation}
\det(I-t\alpha_{\hat{\lambda}}\mid L_0^\Gamma) = \det(I-t\alpha^*_{\hat{\lambda}}\mid B),
\end{equation}
so Lemma 10.5 gives the following result.
\begin{corollary}
The unit reciprocal roots of $P_\lambda(t)$ are obtained from the reciprocal roots of $\det(I-t\alpha^*_{\hat{\lambda}}\mid B)$ of $q$-ordinal equal to $1$ by division by $q$.
\end{corollary}

We saw in Section 1 that $P_\lambda(t)$ has exactly $M$ unit roots when $\lambda\in({\mathbb F}_q^\times)^M\times{\mathbb F}_q^{N-M}$ and $\bar{D}(\lambda)\neq 0$.  By Proposition~8.14, the eigenvalues of the $(M\times M)$-matrix ${\mathcal G}(\lambda)$ are units for all $\lambda\in{\mathcal D}$.  So to prove Theorem~9.8, it suffices to establish the following result.
\begin{proposition}
Let $\lambda\in ({\mathbb F}_q^\times)^M\times{\mathbb F}_q^{N-M}$ and let $\hat{\lambda}\in{\mathbb Q}_p(\zeta_{q-1})^N$ be its Teichm\"uller lifting.  Assume that $\bar{D}(\lambda)\neq 0$.  Then $\hat{\lambda}^{p^i}\in{\mathcal D}$ for $i=0,\dots,a-1$ and $\det\big( I-qt{\mathcal G}(\hat{\lambda}^{p^{a-1}}){\mathcal G}(\hat{\lambda}^{p^{a-1}})\cdots{\mathcal G}(\hat{\lambda})\big)$ is a factor of $\det(I-t\alpha^*_{\hat{\lambda}}\mid B)$.
\end{proposition}

\begin{proof}
Using the notation of (8.15), we have by (8.11)
\begin{equation}
\alpha^*\big({\mathcal G}(\Lambda,x)\big) = p{\mathcal G}(\Lambda){\mathcal G}(\Lambda,x).
\end{equation}
Iterating this gives for all $m\geq 0$
\begin{equation}
(\alpha^*)^m\big({\mathcal G}(\Lambda,x)\big) = p^m{\mathcal G}(\Lambda^{p^{m-1}}){\mathcal G}(\Lambda^{p^{m-2}})\cdots{\mathcal G}(\Lambda){\mathcal G}(\Lambda,x).
\end{equation}
Evaluate ${\mathcal G}(\Lambda,x)$ at $\Lambda = \hat{\lambda}$:
\[ {\mathcal G}(\hat{\lambda},x) = \big( {\mathcal G}^{(1)}(\hat{\lambda},x),\cdots,{\mathcal G}^{(M)}(\hat{\lambda},x)\big), \]
where by (8.16)
\begin{align*}
{\mathcal G}^{(i)}(\hat{\lambda},x) &= \sum_{u\in{\mathcal M}_-} {\mathcal G}_u^{(i)}(\hat{\lambda})\gamma_0^{u_{n+1}} x^u \\
 &= \sum_{u\in{\mathcal M}_-} \big(\gamma_0^{-(p-1)u_{n+1}}{\mathcal G}_u^{(i)}(\hat{\lambda})\big)\gamma_0^{pu_{n+1}}x^u.
\end{align*}
Since $\gamma_0^{-(p-1)u_{n+1}}\to 0$ as $u\to\infty$, this expression lies in $B$.  One checks that the specialization of the left-hand side of (10.11) with $m=a$ at $\Lambda = \hat{\lambda}$ is $\alpha_{\hat{\lambda}}^*({\mathcal G}(\hat{\lambda},x)\big)$, so specializing (10.11) with $m=a$ and $\Lambda = \hat{\lambda}$ gives
\begin{equation}
\alpha^*_{\hat{\lambda}}\big({\mathcal G}(\hat{\lambda},x)\big) = q{\mathcal G}(\hat{\lambda}^{p^{a-1}}){\mathcal G}(\hat{\lambda}^{p^{a-2}})\cdots{\mathcal G}(\hat{\lambda}){\mathcal G}(\hat{\lambda},x).
\end{equation}

We have proved that ${\mathcal G}(\hat{\lambda},x)$ is a vector of $M$ elements of $B$ and that the action of $\alpha^*_{\hat{\lambda}}$ on these $M$ elements is represented by the matrix
\[ q{\mathcal G}(\hat{\lambda}^{p^{a-1}}){\mathcal G}(\hat{\lambda}^{p^{a-2}})\cdots{\mathcal G}(\hat{\lambda}). \]
This implies the assertion of the proposition.
\end{proof}

\end{document}